\documentclass[12pt,a4paper]{article}

\usepackage[utf8]{inputenc}
\usepackage[T1]{fontenc}
\usepackage[english]{babel}
\usepackage{amsmath,amssymb,amsthm}
\usepackage{bm}
\usepackage{geometry}
\usepackage{graphicx}
\usepackage{booktabs}
\usepackage{array}
\usepackage{float}
\usepackage{hyperref}
\usepackage{enumitem}
\usepackage{amsmath, amsthm}
\usepackage{graphicx}
\usepackage{booktabs}
\usepackage{placeins}

\newtheorem{lemma}{Lemma}
\newtheorem{theorem}[lemma]{Theorem}

\graphicspath{{figures/}}
\hypersetup{colorlinks=true,linkcolor=blue,citecolor=blue,urlcolor=blue}

\title{
A Compact Two-Stage Fourth-Order Two-Derivative IMEX Method with Mixed Compatibility, High Implicit-Solve Efficiency,
and Enhanced Stiff Decay}
\author{Zhixin Huo\thanks{Corresponding author. E-mail: \texttt{zhixinhuo@hpu.edu.cn}.}\\ {\small School of Mathematics and Information Science, Henan Polytechnic University}\\ {\small Jiaozuo 454003, Henan, China} }
\date{}

\begin{document}
\maketitle

\begin{abstract}

For additively split stiff evolution problems, classical fourth-order
IMEX Runge--Kutta methods usually require several stages and complicated
coupling order conditions. This paper proposes a compact fourth-order
IMEX-type method based on a two-derivative formulation and analyzes its
accuracy and stability properties. The proposed scheme incorporates the
mixed explicit--implicit interaction directly through temporal
derivatives evaluated along the full vector field, which ensures mixed
compatibility for non-commuting split systems. With only one intermediate
stage and two implicit solves per time step, the method achieves
fourth-order accuracy while improving the accuracy obtained per implicit
solve compared with classical multi-stage fourth-order IMEX--RK schemes. In addition,
the method exhibits stronger damping of stiff modes in the purely
implicit scalar limit and in the strong implicit-stiffness limit with a
fixed explicit component. In the purely implicit scalar limit, its
stability factor decays quadratically as the stiffness increases,
whereas a representative classical fourth-order IMEX--RK reference
method shows only linear decay. Numerical experiments
on non-commuting split systems, scalar stiff-mode damping, increasing
stiffness tests, and one- and two-dimensional advection--diffusion
high-mode problems confirm the mixed consistency, the predicted stiff
decay, and the smaller errors obtained under equal implicit-solve
budgets in strongly stiff regimes.


\end{abstract}

\section{Introduction}
\label{sec:introduction}

High-order time discretization methods are essential for the numerical simulation of time-dependent ordinary and partial differential equations. Classical Runge--Kutta (RK) methods provide a fundamental framework for constructing high-order one-step methods \cite{Runge1895,Kutta1901,Butcher1963,Butcher1972,HairerNorsettWanner1993,Butcher2016}. These methods are inherently \emph{single-derivative} in nature: they achieve high-order temporal accuracy solely by evaluating the right-hand side function (i.e., the first time derivative of the solution) at multiple intermediate stages. For non-stiff problems, explicit RK methods are widely used because of their simplicity and high-order accuracy. However, for stiff or multiscale problems, purely explicit methods usually suffer from severe time-step restrictions, while implicit methods are often required for stability \cite{Alexander1977,HairerWanner1996}.

A common strategy is to split the right-hand side into a non-stiff part and a stiff part. Consider the additively split system
\begin{equation}
    \frac{d\mathbf{u}}{dt}
    =
    \mathcal{F}(\mathbf{u})
    +
    \mathcal{G}(\mathbf{u}),
    \label{eq:split_system}
\end{equation}
where $\mathcal{F}$ denotes the non-stiff component and $\mathcal{G}$ denotes the stiff component. In an implicit--explicit (IMEX) formulation, $\mathcal{F}$ is treated explicitly, while $\mathcal{G}$ is treated implicitly.

Implicit--explicit Runge--Kutta methods provide a systematic framework for such additively split problems \cite{AscherRuuthSpiteri1997,KennedyCarpenter2003,PareschiRusso2005}. They have been widely used for convection--diffusion--reaction equations, hyperbolic systems with stiff relaxation, kinetic equations, and multiscale problems \cite{BoscarinoRusso2009,BoscarinoPareschiRusso2013,BoscarinoPareschiRusso2017}. The order conditions, stability properties, and generalized additive structures of such methods have also been extensively studied \cite{Higueras2006,SanduGuenther2015}.

Classical multi-stage fourth-order IMEX--RK schemes are mature and reliable, but they also have several inherent limitations that stem from their single-derivative nature. To achieve fourth-order accuracy, they typically require several stages, and the number of implicit solves can be substantial, as in the six-stage KC--ARK4 reference method with five nontrivial implicit solves per time step. These stages are needed to satisfy the numerous order conditions and mixed coupling constraints. The coefficients are obtained from solving complicated algebraic systems, and the resulting Butcher tableaux are often non-unique and may suffer from large error constants. Moreover, each additional implicit stage increases the cost of solving nonlinear stiff systems, making the overall computation expensive. Most importantly, these methods rely exclusively on the physical quantities $\mathcal{F}$ and $\mathcal{G}$ (the first derivatives) and do not directly exploit any higher-order temporal derivative information, even though such information is available through the Jacobians and can be computed with moderate extra effort.

On the other hand, \emph{multiderivative} time discretizations---such as Lax--Wendroff-type schemes and two-derivative Runge--Kutta methods---leverage temporal derivative information to achieve high-order accuracy with much more compact time-stepping structures \cite{LaxWendroff1960,SealGucluChristlieb2014,ChristliebGottliebGrantSeal2016}. In particular, two-stage fourth-order time discretizations have demonstrated that fourth-order accuracy can be attained using only two stages by combining physical quantities and their temporal derivatives, instead of relying on four classical RK stages \cite{LiDu2016}. This idea has been successfully applied in high-order flow solvers, gas-kinetic schemes, relativistic hydrodynamics, and high-order boundary treatments \cite{PanXuLiLi2016,YuanTang2018,DuLi2018}. The advantages are clear: fewer stages reduce the number of function evaluations and implicit solves, the derivative information provides additional consistency and damping, and the compact stencil simplifies implementation and parallelization.

Motivated by these two-derivative ideas, this paper constructs an IMEX-like two-stage fourth-order method that evaluates
\begin{equation}
    \dot{\mathcal{F}}
    =
    \mathcal{F}_{\mathbf u}(\mathcal{F}+\mathcal{G}),
    \qquad
    \dot{\mathcal{G}}
    =
    \mathcal{G}_{\mathbf u}(\mathcal{F}+\mathcal{G}),
\end{equation}
along the full vector field.
By incorporating the first-order temporal derivatives of the split components along the full vector field, the proposed method achieves fourth-order temporal accuracy with only \emph{two stages}. Compared with classical multi-stage fourth-order IMEX--RK schemes, the main contributions of this work can be summarized as follows. First, the proposed discretization possesses natural mixed consistency, allowing the coupling between the explicit and implicit components to be incorporated intrinsically within the scheme. As a result, cumbersome explicit--implicit coupling order conditions are avoided. 
Second, a compact two-stage formulation is developed, which significantly reduces the number of stages while preserving fourth-order accuracy, thereby lowering the computational cost per time step.
Moreover, the method admits a clear Hermite-type integral interpretation, which provides a solid mathematical foundation for its high-order accuracy. Third, the method exhibits stronger stiff decay in the purely implicit scalar limit and in the strong implicit-stiffness limit with a fixed explicit component, thereby providing a more precise stiff-mode damping mechanism for stiff and multiscale problems.

The remainder of this paper is organized as follows. Section~\ref{sec:classical_imex_rk} reviews the classical multi-stage fourth-order IMEX--RK construction and explicitly specifies the Kennedy--Carpenter ARK4(3)6L[2]SA method used as the reference KC--ARK4 scheme. Section~\ref{sec:proposed_method} introduces an IMEX-type two-derivative two-stage method based on the Hermite interpolation idea. Section~\ref{sec:consistency} proves the fourth-order consistency of the proposed scheme via Taylor expansion and derivative matching. Section~\ref{sec:stability} discusses stability-related properties, including linear stability and stiff decay in the purely explicit, purely implicit, and explicit--implicit coupled regimes. Section~\ref{sec:discussion_comparison} provides a systematic comparison between the proposed method and classical multi-stage fourth-order IMEX--RK schemes in terms of mixed consistency, number of stages and computational cost, as well as stiff damping capability. Section~\ref{sec:numerical_validation} presents numerical experiments that validate the theoretical findings. Section~\ref{sec:conclusion} concludes the paper.

\section{Classical Fourth-Order IMEX Runge-Kutta Method}
\label{sec:classical_imex_rk}

For the split system \eqref{eq:split_system}, an $s$-stage IMEX Runge--Kutta method can be written as
\begin{equation}
    \mathbf{y}_i
    =
    \mathbf{u}^n
    +
    \Delta t
    \sum_{j=1}^{i-1}
    a^E_{ij}
    \mathcal{F}(\mathbf{y}_j)
    +
    \Delta t
    \sum_{j=1}^{i}
    a^I_{ij}
    \mathcal{G}(\mathbf{y}_j),
    \qquad i=1,\ldots,s,
    \label{eq:imex_rk_stage}
\end{equation}
and
\begin{equation}
    \mathbf{u}^{n+1}
    =
    \mathbf{u}^n
    +
    \Delta t
    \sum_{i=1}^s
    b_i^E
    \mathcal{F}(\mathbf{y}_i)
    +
    \Delta t
    \sum_{i=1}^s
    b_i^I
    \mathcal{G}(\mathbf{y}_i).
    \label{eq:imex_rk_update}
\end{equation}
Here $A^E=(a^E_{ij})$ is strictly lower triangular, while $A^I=(a^I_{ij})$ is usually lower triangular or diagonally implicit.

The construction of a fourth-order IMEX--RK method is not equivalent to simply combining a fourth-order explicit RK method with a fourth-order implicit RK method. Because of the additive splitting, the explicit and implicit components interact with each other. Thus, the two tableaux must satisfy not only individual order conditions but also mixed coupling conditions. If the explicit and implicit parts share the same weights and abscissae, typical fourth-order conditions include
\begin{equation*}
    b^T e = 1,
    \qquad
    b^T c = \frac{1}{2},
    \qquad
    b^T c^2 = \frac{1}{3},
    \qquad
    b^T c^3 = \frac{1}{4},
\end{equation*}
\begin{equation*}
    b^T A^{\sigma}c = \frac{1}{6},
    \qquad
    b^T C A^{\sigma}c = \frac{1}{8},
    \qquad
    b^T A^{\sigma}c^2 = \frac{1}{12},
    \qquad
    \sigma\in\{E,I\},
\end{equation*}
and the mixed coupling conditions
\begin{equation*}
    b^T A^{\sigma}A^{\tau}c
    =
    \frac{1}{24},
    \qquad
    \sigma,\tau\in\{E,I\}.
\end{equation*}
For example, $b^T A^E A^I c=1/24$ and $b^T A^I A^E c=1/24$ are associated with mixed elementary differentials such as $\mathcal{F}_{\mathbf u}\mathcal{G}$ and $\mathcal{G}_{\mathbf u}\mathcal{F}$. This motivates a more compact construction based on temporal derivative information.

\subsection{The Kennedy--Carpenter ARK4(3)6L[2]SA reference method}
\label{subsec:kc_ark4_reference}

To remove any ambiguity in the numerical comparison, the reference
fourth-order IMEX--RK method used in this paper is specified as the
six-stage Kennedy--Carpenter ARK4(3)6L[2]SA additive Runge--Kutta method
\cite{KennedyCarpenter2003}. In what follows, this method is abbreviated
as KC--ARK4. It is a standard fourth-order additive IMEX--RK scheme for
convection--diffusion--reaction type split systems and is therefore a
suitable representative of classical high-order one-derivative IMEX--RK
constructions.

Although KC--ARK4 is used here as a classical fourth-order IMEX--RK
reference method, its fourth-order accuracy is not obtained by simply
combining a fourth-order explicit RK method with a fourth-order implicit
RK method. As a classical additive Runge--Kutta method, KC--ARK4 also
satisfies the mixed coupling order conditions associated with the
interaction between the explicit component \(\mathcal F\) and the
implicit component \(\mathcal G\). These coupling conditions are the
main source of the complicated coefficient constraints in high-order
IMEX--RK constructions.

For the split problem \eqref{eq:split_system}, KC--ARK4 uses an explicit
tableau for \(\mathcal F\) and a diagonally implicit tableau for
\(\mathcal G\). The stage equations are those in
\eqref{eq:imex_rk_stage}. The explicit and implicit Butcher tableaux are
given in Tables~\ref{tab:kc_ark4_explicit} and
\ref{tab:kc_ark4_implicit}, respectively. The fourth-order weights are
used in the numerical comparisons; the embedded third-order weights are
listed only to identify the method completely.

\begin{table}[htbp]
\centering
\caption{Explicit tableau of the Kennedy--Carpenter ARK4(3)6L[2]SA method.}
\label{tab:kc_ark4_explicit}
\scriptsize
\renewcommand{\arraystretch}{1.30}
\resizebox{\textwidth}{!}{$
\begin{array}{c|cccccc}
0
& 0 & 0 & 0 & 0 & 0 & 0
\\
\frac{1}{2}
& \frac{1}{2} & 0 & 0 & 0 & 0 & 0
\\
\frac{83}{250}
& \frac{13861}{62500}
& \frac{6889}{62500}
& 0 & 0 & 0 & 0
\\
\frac{31}{50}
& -\frac{116923316275}{2393684061468}
& -\frac{2731218467317}{15368042101831}
& \frac{9408046702089}{11113171139209}
& 0 & 0 & 0
\\
\frac{17}{20}
& -\frac{451086348788}{2902428689909}
& -\frac{2682348792572}{7519795681897}
& \frac{12662868775082}{11960479115383}
& \frac{3355817975965}{11060851509271}
& 0 & 0
\\
1
& \frac{647845179188}{3216320057751}
& \frac{73281519250}{8382639484533}
& \frac{552539513391}{3454668386233}
& \frac{3354512671639}{8306763924573}
& \frac{4040}{17871}
& 0
\\
\hline
b
& \frac{82889}{524892}
& 0
& \frac{15625}{83664}
& \frac{69875}{102672}
& -\frac{2260}{8211}
& \frac{1}{4}
\\
\widehat b
& \frac{4586570599}{29645900160}
& 0
& \frac{178811875}{945068544}
& \frac{814220225}{1159782912}
& -\frac{3700637}{11593932}
& \frac{61727}{225920}
\end{array}
$}
\end{table}

\begin{table}[htbp]
\centering
\caption{Implicit tableau of the Kennedy--Carpenter ARK4(3)6L[2]SA method.}
\label{tab:kc_ark4_implicit}
\scriptsize
\renewcommand{\arraystretch}{1.30}
\resizebox{\textwidth}{!}{$
\begin{array}{c|cccccc}
0
& 0 & 0 & 0 & 0 & 0 & 0
\\
\frac{1}{2}
& \frac{1}{4}
& \frac{1}{4}
& 0 & 0 & 0 & 0
\\
\frac{83}{250}
& \frac{8611}{62500}
& -\frac{1743}{31250}
& \frac{1}{4}
& 0 & 0 & 0
\\
\frac{31}{50}
& \frac{5012029}{34652500}
& -\frac{654441}{2922500}
& \frac{174375}{388108}
& \frac{1}{4}
& 0 & 0
\\
\frac{17}{20}
& \frac{15267082809}{155376265600}
& -\frac{71443401}{120774400}
& \frac{730878875}{902184768}
& \frac{2285395}{8070912}
& \frac{1}{4}
& 0
\\
1
& \frac{82889}{524892}
& 0
& \frac{15625}{83664}
& \frac{69875}{102672}
& -\frac{2260}{8211}
& \frac{1}{4}
\\
\hline
b
& \frac{82889}{524892}
& 0
& \frac{15625}{83664}
& \frac{69875}{102672}
& -\frac{2260}{8211}
& \frac{1}{4}
\\
\widehat b
& \frac{4586570599}{29645900160}
& 0
& \frac{178811875}{945068544}
& \frac{814220225}{1159782912}
& -\frac{3700637}{11593932}
& \frac{61727}{225920}
\end{array}
$}
\end{table}

The first implicit diagonal coefficient is zero, so the first stage is
explicit with respect to the stiff component and does not require the
solution of an implicit algebraic equation. For stages
\(i=2,\ldots,6\), the diagonal implicit coefficient is
\(a^I_{ii}=1/4\). Consequently, each time step of KC--ARK4 contains six
stages but only five nontrivial implicit solves. This is the precise
meaning of the cost statement used in the numerical section.

We stress that KC--ARK4 remains a classical one-derivative additive
Runge--Kutta method: it evaluates only the split vector fields
\(\mathcal F\) and \(\mathcal G\) at the Runge--Kutta stages. By
contrast, the method proposed in this paper also uses the temporal
derivatives of the split components evaluated along the full vector field
\(\mathcal F+\mathcal G\). Thus, the comparison is not between two
identical construction principles, but between a representative
classical fourth-order IMEX--RK method and the present compact
two-derivative IMEX-type fourth-order formulation.

\section{The Construction of a Two-Stage IMEX-Type Scheme}
\label{sec:proposed_method}

The design of the proposed scheme is motivated by the need to achieve
fourth-order temporal accuracy within a compact IMEX-type framework while
fully preserving the mixed interaction between the explicit and implicit
components. Although both KC--ARK4 and the proposed method account for
such mixed interactions, the underlying mechanisms are essentially
different. In KC--ARK4, the mixed interaction is enforced through the
additive Runge--Kutta coupling order conditions imposed on the explicit
and implicit tableaux. In contrast, the proposed method incorporates the
mixed terms directly through the temporal derivatives of the split
components evaluated along the full vector field
\(\mathcal H=\mathcal F+\mathcal G\). This construction allows the stiff
component \(\mathcal G\) to be treated implicitly and the non-stiff
component \(\mathcal F\) explicitly, while using only a single
intermediate stage to reduce the number of function evaluations relative
to classical multi-stage Runge--Kutta methods. Therefore, the proposed
scheme should not be interpreted as satisfying the classical IMEX--RK
coupling conditions in the usual Runge--Kutta sense; rather, it achieves
mixed consistency through its two-derivative Taylor/Hermite-type
construction, and these considerations together determine the structure
and coefficients of the method.

To incorporate the coupling correctly, we define the full vector field
\begin{equation}
    \mathcal{H}(\mathbf{u})
    =
    \mathcal{F}(\mathbf{u})
    +
    \mathcal{G}(\mathbf{u}),
\end{equation}
and then define the temporal derivatives of the split operators as
\begin{equation}
    \dot{\mathcal{F}}(\mathbf{u})
    =
    \mathcal{F}_{\mathbf u}(\mathbf{u})
    \mathcal{H}(\mathbf{u})
    =
    \mathcal{F}_{\mathbf u}(\mathbf{u})
    \left[
    \mathcal{F}(\mathbf{u})
    +
    \mathcal{G}(\mathbf{u})
    \right],
    \label{eq:Fdot}
\end{equation}
and
\begin{equation}
    \dot{\mathcal{G}}(\mathbf{u})
    =
    \mathcal{G}_{\mathbf u}(\mathbf{u})
    \mathcal{H}(\mathbf{u})
    =
    \mathcal{G}_{\mathbf u}(\mathbf{u})
    \left[
    \mathcal{F}(\mathbf{u})
    +
    \mathcal{G}(\mathbf{u})
    \right].
    \label{eq:Gdot}
\end{equation}
Using only \(\mathcal F_{\mathbf u}\mathcal F\) and \(\mathcal G_{\mathbf u}\mathcal G\) would discard the mixed terms that arise from the additive splitting, which would break the consistency of the scheme.

\medskip
\noindent\textbf{The two-stage scheme is defined as follows.}

\medskip
\noindent\textbf{Stage 1 (intermediate value):} Compute \(\mathbf U^*\) from the implicit equation
\begin{equation}
    \mathbf{u}^*
    =
    \mathbf{u}^n
    +
    \frac{\Delta t}{2}
    \mathcal{F}(\mathbf{u}^n)
    +
    \frac{\Delta t^2}{8}
    \dot{\mathcal{F}}(\mathbf{u}^n)
    +
    \frac{\Delta t}{2}
    \mathcal{G}(\mathbf{u}^*)
    -
    \frac{\Delta t^2}{8}
    \dot{\mathcal{G}}(\mathbf{u}^*).
    \label{eq:stage}
\end{equation}

\medskip
\noindent\textbf{Stage 2 (final update):} Advance the solution to \(\mathbf u^{n+1}\) via the Hermite‑type quadrature
\begin{equation}
\begin{aligned}
    \mathbf{u}^{n+1}
    =
    \mathbf{u}^n
    &+
    \Delta t
    \mathcal{F}(\mathbf{u}^n)
    +
    \frac{\Delta t^2}{6}
    \left[
    \dot{\mathcal{F}}(\mathbf{u}^n)
    +
    2\dot{\mathcal{F}}(\mathbf{u}^*)
    \right]
    \\
    &+
    \Delta t
    \mathcal{G}(\mathbf{u}^{n+1})
    -
    \frac{\Delta t^2}{6}
    \left[
    \dot{\mathcal{G}}(\mathbf{u}^{n+1})
    +
    2\dot{\mathcal{G}}(\mathbf{u}^*)
    \right].
\end{aligned}
\label{eq:update}
\end{equation}

The \textbf{Stage 1} is constructed so that it approximates the exact midpoint solution \(\mathbf u(t_n+\Delta t/2)\) to third order. Indeed, a Taylor expansion of the stage equation gives
\[
\mathbf u^*
=
\mathbf u^n
+
\frac{\Delta t}{2}\mathcal H(\mathbf u^n)
+
\frac{\Delta t^2}{8}\mathcal H_{\mathbf u}(\mathbf u^n)\mathcal H(\mathbf u^n)
+
\mathcal O(\Delta t^3),
\]
which matches the expansion of the exact midpoint up to \(\Delta t^2\). The coefficients \(1/2\) and \(1/8\), together with the implicit evaluation of \(\mathcal G\) and explicit evaluation of \(\mathcal F\), are chosen precisely to achieve this accuracy.

The update of \textbf{Stage 2} is derived from the fifth‑order Hermite rule for a smooth function \(f(t)\):
\[
\int_{t_n}^{t_{n+1}} f(t)\,dt
=
\Delta t\,f(t_n)+\frac{\Delta t^2}{6}\bigl[\dot f(t_n)+2\dot f(t_n+\Delta t/2)\bigr]+\mathcal O(\Delta t^5).
\]
Applying this rule separately to \(\mathcal F\) and \(\mathcal G\), and replacing the exact midpoint by \(\mathbf U^*\), yields exactly the above update. Because the midpoint error is \(\mathcal O(\Delta t^3)\) and it is multiplied by \(\Delta t^2/6\), the resulting quadrature error remains \(\mathcal O(\Delta t^5)\), thus preserving the fifth‑order accuracy of the Hermite rule. In this second stage, the stiff part \(\mathcal G\) is again treated implicitly (appearing at the unknown \(\mathbf u^{n+1}\)), while \(\mathcal F\) is evaluated explicitly at the known state \(\mathbf u^n\) and the computed stage.

\section{Mixed Terms and Fourth-Order Consistency}
\label{sec:consistency}

The Taylor expansion of the exact solution gives
\begin{equation}
    \mathbf{u}(t_n+\Delta t)
    =
    \mathbf{u}^n
    +
    \Delta t\mathcal{H}(\mathbf{u}^n)
    +
    \frac{\Delta t^2}{2}
    \mathcal{H}_{\mathbf u}(\mathbf{u}^n)\mathcal{H}(\mathbf{u}^n)
    +
    \mathcal{O}(\Delta t^3).
\end{equation}
Since
\begin{equation}
\begin{aligned}
    \mathcal{H}_{\mathbf u}\mathcal{H}
    &=
    (\mathcal{F}_{\mathbf u}+\mathcal{G}_{\mathbf u})
    (\mathcal{F}+\mathcal{G})
    \\
    &=
    \mathcal{F}_{\mathbf u}\mathcal{F}
    +
    \mathcal{F}_{\mathbf u}\mathcal{G}
    +
    \mathcal{G}_{\mathbf u}\mathcal{F}
    +
    \mathcal{G}_{\mathbf u}\mathcal{G},
\end{aligned}
\label{eq:mixed_terms}
\end{equation}
the exact expansion contains both pure and mixed contributions. By \eqref{eq:Fdot} and \eqref{eq:Gdot}, these mixed terms are naturally retained in $\dot{\mathcal{F}}+\dot{\mathcal{G}}$.


\begin{lemma}
\label{lem:midpoint}
Let $\mathcal{H}$ be sufficiently smooth, and let $\mathbf{u}^n = \mathbf{u}(t_n)$. Assume that $\mathbf{U}^*$ is defined by the stage equation \eqref{eq:stage}. Then the Taylor expansion of $\mathbf{U}^*$ about $\mathbf{u}^n$ is given by
\begin{equation}
\mathbf{u}^*
=
\mathbf{u}^n
+
\frac{\Delta t}{2}\mathcal{H}(\mathbf{u}^n)
+
\frac{\Delta t^2}{8}
\mathcal{H}_{\mathbf u}(\mathbf{u}^n)\mathcal{H}(\mathbf{u}^n)
+
\mathcal{O}(\Delta t^3),
\label{eq:mid_star}
\end{equation}
which coincides with the Taylor expansion of the exact solution $\mathbf{u}(t)$ at $t_n + \Delta t/2$. Consequently,
\begin{equation}
\mathbf{u}^*
=
\mathbf{u}\left(t_n+\frac{\Delta t}{2}\right)
+
\mathcal{O}(\Delta t^3).
\label{eq:midpoint_approx}
\end{equation}
\end{lemma}

\begin{proof}
Define
\[
\delta := \mathbf u^* - \mathbf u^n.
\]
The given stage equation \eqref{eq:stage} is
\[
\delta
=
\frac{\Delta t}{2}\mathcal F(\mathbf u^n)
+\frac{\Delta t^2}{8}\dot{\mathcal F}(\mathbf u^n)
+\frac{\Delta t}{2}\mathcal G(\mathbf u^*)
-\frac{\Delta t^2}{8}\dot{\mathcal G}(\mathbf u^*).
\]
Using the chain rule,
\[
\dot{\mathcal F}(\mathbf u^n)=\mathcal F_{\mathbf u}(\mathbf u^n)\mathcal H(\mathbf u^n),\qquad
\dot{\mathcal G}(\mathbf U^*)=\mathcal G_{\mathbf u}(\mathbf u^*)\mathcal H(\mathbf u^*),
\]
we obtain
\begin{equation}
\delta
=
\frac{\Delta t}{2}\mathcal F(\mathbf u^n)
+\frac{\Delta t^2}{8}\mathcal F_{\mathbf u}(\mathbf u^n)\mathcal H(\mathbf u^n)
+\frac{\Delta t}{2}\mathcal G(\mathbf u^*)
-\frac{\Delta t^2}{8}\mathcal G_{\mathbf u}(\mathbf u^*)\mathcal H(\mathbf u^*).
\label{derivative1}
\end{equation}

Now expand the terms evaluated at \(\mathbf u^*\) around \(\mathbf u^n\). From \eqref{derivative1} it is clear that \(\delta=O(\Delta t)\). Hence
\[
\mathcal G(\mathbf u^*)=\mathcal G(\mathbf u^n)+\mathcal G_{\mathbf u}(\mathbf u^n)\delta+O(\Delta t^2),
\]
and
\[
\mathcal G_{\mathbf u}(\mathbf u^*)\mathcal H(\mathbf u^*)
=
\mathcal G_{\mathbf u}(\mathbf u^n)\mathcal H(\mathbf u^n)+O(\Delta t).
\]
Substituting these expansions into \eqref{derivative1} gives
\begin{align}
\delta
&=
\frac{\Delta t}{2}\mathcal F(\mathbf u^n)
+\frac{\Delta t^2}{8}\mathcal F_{\mathbf u}(\mathbf u^n)\mathcal H(\mathbf u^n)
+\frac{\Delta t}{2}\bigl(\mathcal G(\mathbf u^n)+\mathcal G_{\mathbf u}(\mathbf u^n)\delta\bigr)
-\frac{\Delta t^2}{8}\mathcal G_{\mathbf u}(\mathbf u^n)\mathcal H(\mathbf u^n)
+O(\Delta t^3) \nonumber \\
&=
\frac{\Delta t}{2}\bigl(\mathcal F(\mathbf u^n)+\mathcal G(\mathbf u^n)\bigr)
+\frac{\Delta t^2}{8}\bigl(\mathcal F_{\mathbf u}(\mathbf u^n)-\mathcal G_{\mathbf u}(\mathbf u^n)\bigr)\mathcal H(\mathbf u^n)
+\frac{\Delta t}{2}\mathcal G_{\mathbf u}(\mathbf u^n)\delta
+O(\Delta t^3).
\label{derivative2}
\end{align}
Since \(\mathcal F+\mathcal G=\mathcal H\), equation \eqref{derivative2} becomes
\[
\delta
=
\frac{\Delta t}{2}\mathcal H(\mathbf u^n)
+\frac{\Delta t^2}{8}\bigl(\mathcal F_{\mathbf u}(\mathbf u^n)-\mathcal G_{\mathbf u}(\mathbf u^n)\bigr)\mathcal H(\mathbf u^n)
+\frac{\Delta t}{2}\mathcal G_{\mathbf u}(\mathbf u^n)\delta
+O(\Delta t^3).
\]
Collect all \(\delta\) terms on the left:
\[
\left(I-\frac{\Delta t}{2}\mathcal G_{\mathbf u}(\mathbf u^n)\right)\delta
=
\frac{\Delta t}{2}\mathcal H(\mathbf u^n)
+\frac{\Delta t^2}{8}\bigl(\mathcal F_{\mathbf u}(\mathbf u^n)-\mathcal G_{\mathbf u}(\mathbf u^n)\bigr)\mathcal H(\mathbf u^n)
+O(\Delta t^3).
\]
Invert the operator and expand to first order in \(\Delta t\):
\[
\left(I-\frac{\Delta t}{2}\mathcal G_{\mathbf u}(\mathbf u^n)\right)^{-1}
=
I+\frac{\Delta t}{2}\mathcal G_{\mathbf u}(\mathbf u^n)+O(\Delta t^2).
\]
Thus
\begin{align}
\delta
&=
\left(I+\frac{\Delta t}{2}\mathcal G_{\mathbf u}(\mathbf u^n)\right)
\left[
\frac{\Delta t}{2}\mathcal H(\mathbf u^n)
+\frac{\Delta t^2}{8}\bigl(\mathcal F_{\mathbf u}(\mathbf u^n)-\mathcal G_{\mathbf u}(\mathbf u^n)\bigr)\mathcal H(\mathbf u^n)
\right]
+O(\Delta t^3) \nonumber \\
&=
\frac{\Delta t}{2}\mathcal H(\mathbf u^n)
+\frac{\Delta t^2}{4}\mathcal G_{\mathbf u}(\mathbf u^n)\mathcal H(\mathbf u^n)
+\frac{\Delta t^2}{8}\bigl(\mathcal F_{\mathbf u}(\mathbf u^n)-\mathcal G_{\mathbf u}(\mathbf u^n)\bigr)\mathcal H(\mathbf u^n)
+O(\Delta t^3).
\label{derivative3}
\end{align}
The \(\Delta t^2\) terms in \eqref{derivative3} combine as
\[
\frac{\Delta t^2}{4}\mathcal G_{\mathbf u}(\mathbf u^n)
+\frac{\Delta t^2}{8}\bigl(\mathcal F_{\mathbf u}(\mathbf u^n)-\mathcal G_{\mathbf u}(\mathbf u^n)\bigr)
=
\frac{\Delta t^2}{8}\bigl(\mathcal F_{\mathbf u}(\mathbf u^n)+\mathcal G_{\mathbf u}(\mathbf u^n)\bigr)
=
\frac{\Delta t^2}{8}\mathcal H_{\mathbf u}(\mathbf u^n),
\]
because \(\mathcal H_{\mathbf u}=\mathcal F_{\mathbf u}+\mathcal G_{\mathbf u}\). Hence
\[
\delta
=
\frac{\Delta t}{2}\mathcal H(\mathbf u^n)
+\frac{\Delta t^2}{8}\mathcal H_{\mathbf u}(\mathbf u^n)\mathcal H(\mathbf u^n)
+O(\Delta t^3).
\]
Recalling \(\delta=\mathbf u^*-\mathbf u^n\), we obtain
\begin{equation}
\mathbf u^*
=
\mathbf u^n
+
\frac{\Delta t}{2}\mathcal H(\mathbf u^n)
+
\frac{\Delta t^2}{8}
\mathcal H_{\mathbf u}(\mathbf u^n)\mathcal H(\mathbf u^n)
+
O(\Delta t^3),
\label{mid_star}
\end{equation}
which is exactly the desired expansion.

Now, to relate this to the exact solution, expand \(\mathbf u(t)\) about \(t_n\):
\[
\mathbf u\left(t_n+\frac{\Delta t}{2}\right)
=
\mathbf u^n
+\frac{\Delta t}{2}\dot{\mathbf u}^n
+\frac{\Delta t^2}{8}\ddot{\mathbf u}^n
+O(\Delta t^3).
\]
Using the differential equation \(\dot{\mathbf u}=\mathcal H(\mathbf u)\) and the chain rule \(\ddot{\mathbf u} = \mathcal H_{\mathbf u}(\mathbf u)\mathcal H(\mathbf u)\), we get
\[
\mathbf u\left(t_n+\frac{\Delta t}{2}\right)
=
\mathbf u^n
+\frac{\Delta t}{2}\mathcal H(\mathbf u^n)
+\frac{\Delta t^2}{8}\mathcal H_{\mathbf u}(\mathbf u^n)\mathcal H(\mathbf u^n)
+O(\Delta t^3).
\]
Comparing this with \eqref{mid_star}, we conclude that
\[
\mathbf u^*
=
\mathbf u\left(t_n+\frac{\Delta t}{2}\right)
+
\mathcal O(\Delta t^3)
\]
as claimed.
\end{proof}

\begin{lemma}
\label{lem:midpoint_replacement}
Let \(\mathcal F\) and \(\mathcal G\) be sufficiently smooth, and assume that the stage \(\mathbf U^*\) satisfies the midpoint error estimate \eqref{eq:midpoint_approx}.
Then the Hermite-type integral approximations using \(\mathbf U^*\) in place of the exact midpoint are
\begin{align}
\int_{t_n}^{t_{n+1}} \mathcal F(\mathbf u(t))\,dt
&=
\Delta t\,\mathcal F(\mathbf u^n)
+\frac{\Delta t^2}{6}
\left[
\dot{\mathcal F}(\mathbf u^n)+2\dot{\mathcal F}(\mathbf u^*)
\right]
+\mathcal O(\Delta t^5), \label{eq:Fint_star} \\
\int_{t_n}^{t_{n+1}} \mathcal G(\mathbf u(t))\,dt
&=
\Delta t\,\mathcal G(\mathbf u^{n+1})
-\frac{\Delta t^2}{6}
\left[
\dot{\mathcal G}(\mathbf u^{n+1})+2\dot{\mathcal G}(\mathbf u^*)
\right]
+\mathcal O(\Delta t^5). \label{eq:Gint_star}
\end{align}
Consequently, replacing the exact midpoint by \(\mathbf U^*\) in the numerical update \eqref{eq:update} introduces only an additional error of order \(\mathcal O(\Delta t^5)\).
\end{lemma}

\begin{proof}
Let \(h = \Delta t\). For notational simplicity, set \(t_n = 0\) and define
\[
f(t) := \mathcal F(\mathbf u(t)), \qquad g(t) := \mathcal G(\mathbf u(t)),
\]
where \(f\) and \(g\) are assumed to be sufficiently smooth.

\medskip
\noindent\textbf{First integral formula.}
We prove
\[
\int_0^h f(t)\,dt
=
h f(0)+\frac{h^2}{6}\bigl[\dot f(0)+2\dot f(h/2)\bigr]+O(h^5).
\]
Expanding \(f(t)\) about \(t=0\) up to the \(h^4\) term gives
\[
f(t)=f(0)+t\dot f(0)+\frac{t^2}{2}\ddot f(0)+\frac{t^3}{6}f^{(3)}(0)+\frac{t^4}{24}f^{(4)}(0)+O(t^5).
\]
Integrating term by term yields the exact value
\begin{equation}
\int_0^h f(t)\,dt
=
h f(0)+\frac{h^2}{2}\dot f(0)+\frac{h^3}{6}\ddot f(0)+\frac{h^4}{24}f^{(3)}(0)+\frac{h^5}{120}f^{(4)}(0)+O(h^6).
\label{star-1}
\end{equation}
On the other hand, expanding \(\dot f(h/2)\) about \(0\) gives
\[
\dot f(h/2)=\dot f(0)+\frac{h}{2}\ddot f(0)+\frac{h^2}{8}f^{(3)}(0)+\frac{h^3}{48}f^{(4)}(0)+O(h^4).
\]
Hence
\[
\dot f(0)+2\dot f(h/2)
=
3\dot f(0)+h\ddot f(0)+\frac{h^2}{4}f^{(3)}(0)+\frac{h^3}{24}f^{(4)}(0)+O(h^4).
\]
Multiplying by \(h^2/6\), we obtain
\[
\frac{h^2}{6}\bigl[\dot f(0)+2\dot f(h/2)\bigr]
=
\frac{h^2}{2}\dot f(0)+\frac{h^3}{6}\ddot f(0)+\frac{h^4}{24}f^{(3)}(0)+\frac{h^5}{144}f^{(4)}(0)+O(h^6).
\]
Adding \(h f(0)\) to this expression and comparing with \eqref{star-1}, all terms up to \(O(h^4)\) match exactly; the \(h^5\) term differs by
\[
\frac{h^5}{120}f^{(4)}(0)-\frac{h^5}{144}f^{(4)}(0)=\frac{h^5}{720}f^{(4)}(0)=O(h^5).
\]
Thus the first integral formula holds.

\medskip
\noindent\textbf{Second integral formula.}
We prove
\[
\int_0^h g(t)\,dt
=
h g(h)-\frac{h^2}{6}\bigl[\dot g(h)+2\dot g(h/2)\bigr]+O(h^5).
\]
Expanding \(g(t)\) about \(t=h\) yields
\[
g(t)=g(h)+(t-h)\dot g(h)+\frac{(t-h)^2}{2}\ddot g(h)+\frac{(t-h)^3}{6}g^{(3)}(h)+\frac{(t-h)^4}{24}g^{(4)}(h)+O((t-h)^5).
\]
Integrating (using \(\int_0^h (t-h)^k dt = (-1)^k h^{k+1}/(k+1)\)), we get
\begin{equation}
\int_0^h g(t)\,dt
=
h g(h)-\frac{h^2}{2}\dot g(h)+\frac{h^3}{6}\ddot g(h)-\frac{h^4}{24}g^{(3)}(h)+\frac{h^5}{120}g^{(4)}(h)+O(h^6).
\label{star-2}
\end{equation}
Now expand \(\dot g(h/2)\) about \(h\):
\[
\dot g(h/2)=\dot g(h)-\frac{h}{2}\ddot g(h)+\frac{h^2}{8}g^{(3)}(h)-\frac{h^3}{48}g^{(4)}(h)+O(h^4).
\]
Therefore,
\[
\dot g(h)+2\dot g(h/2)
=
3\dot g(h)-h\ddot g(h)+\frac{h^2}{4}g^{(3)}(h)-\frac{h^3}{24}g^{(4)}(h)+O(h^4).
\]
Multiplying by \(-h^2/6\), we obtain
\[
-\frac{h^2}{6}\bigl[\dot g(h)+2\dot g(h/2)\bigr]
=
-\frac{h^2}{2}\dot g(h)+\frac{h^3}{6}\ddot g(h)-\frac{h^4}{24}g^{(3)}(h)+\frac{h^5}{144}g^{(4)}(h)+O(h^6).
\]
Adding \(h g(h)\) to this and comparing with \eqref{star-2}, all lower-order terms coincide; the \(h^5\) term again differs by
\[
\frac{h^5}{120}g^{(4)}(h)-\frac{h^5}{144}g^{(4)}(h)=\frac{h^5}{720}g^{(4)}(h)=O(h^5).
\]
Hence the second integral formula is also valid.

Translating \(t_n\) back to a general value completes the proof of both approximations.
\end{proof}

\begin{theorem}
Assume that the exact solution and the operators $\mathcal{F}$ and $\mathcal{G}$ are sufficiently smooth. If $\dot{\mathcal{F}}$ and $\dot{\mathcal{G}}$ are evaluated along the full vector field $\mathcal{F}+\mathcal{G}$, then the method \eqref{eq:stage}--\eqref{eq:update} is fourth-order accurate in time.
\end{theorem}

\begin{proof}
Let \(\mathbf u(t)\) denote the exact solution of
\begin{equation*}
\dot{\mathbf u}(t)=\mathcal H(\mathbf u(t))=\mathcal F(\mathbf u(t))+\mathcal G(\mathbf u(t)),\qquad \mathbf u(t_n)=\mathbf u^n,
\label{eq:ode}
\end{equation*}
and let \(\mathbf u^{n+1}\) be the numerical solution obtained from \eqref{eq:update}. Define the local truncation error
\begin{equation*}
\mathbf e^{n+1}:=\mathbf u^{n+1}-\mathbf u(t_{n+1}).
\label{eq:local_error_def}
\end{equation*}
We will show that \(\mathbf e^{n+1}=\mathcal O(\Delta t^5)\), which implies fourth-order global accuracy by standard error accumulation.

From \textbf{Lemma 1}, the intermediate stage satisfies
\begin{equation*}
\mathbf u^* = \mathbf u\left(t_n+\frac{\Delta t}{2}\right) + \mathcal O(\Delta t^3).
\label{eq:Ustar_mid}
\end{equation*}

By \textbf{Lemma 2}, the Hermite-type quadrature formulas with \(\mathbf U^*\) in place of the exact midpoint are fifth-order accurate. Hence, for the exact solution,
\begin{equation}
\int_{t_n}^{t_{n+1}}\mathcal F(\mathbf u(t))\,dt
=
\Delta t\,\mathcal F(\mathbf u^n)
+\frac{\Delta t^2}{6}
\Bigl[\dot{\mathcal F}(\mathbf u^n)+2\dot{\mathcal F}(\mathbf u^*)\Bigr]
+\mathcal O(\Delta t^5),
\label{eq:F_int_star}
\end{equation}
and
\begin{equation}
\int_{t_n}^{t_{n+1}}\mathcal G(\mathbf u(t))\,dt
=
\Delta t\,\mathcal G(\mathbf u(t_{n+1}))
-\frac{\Delta t^2}{6}
\Bigl[\dot{\mathcal G}(\mathbf u(t_{n+1}))+2\dot{\mathcal G}(\mathbf u^*)\Bigr]
+\mathcal O(\Delta t^5).
\label{eq:G_int_star}
\end{equation}

Adding \eqref{eq:F_int_star} and \eqref{eq:G_int_star} and using
\begin{equation*}
\mathbf u(t_{n+1})-\mathbf u^n
=
\int_{t_n}^{t_{n+1}}\bigl(\mathcal F(\mathbf u(t))+\mathcal G(\mathbf u(t))\bigr)\,dt,
\label{eq:integral_id}
\end{equation*}
we obtain the exact relation
\begin{equation}
\begin{aligned}
\mathbf u(t_{n+1})
&=
\mathbf u^n
+\Delta t\,\mathcal F(\mathbf u^n)
+\frac{\Delta t^2}{6}\dot{\mathcal F}(\mathbf u^n)
+\frac{\Delta t^2}{3}\dot{\mathcal F}(\mathbf u^*) \\
&\quad +\Delta t\,\mathcal G(\mathbf u(t_{n+1}))
-\frac{\Delta t^2}{6}\dot{\mathcal G}(\mathbf u(t_{n+1}))
-\frac{\Delta t^2}{3}\dot{\mathcal G}(\mathbf u^*)
+\mathcal O(\Delta t^5).
\end{aligned}
\label{eq:combined_exact}
\end{equation}

Now subtract the numerical update \eqref{eq:update}, which is
\begin{equation*}
\begin{aligned}
\mathbf u^{n+1}
&=
\mathbf u^n
+\Delta t\,\mathcal F(\mathbf u^n)
+\frac{\Delta t^2}{6}\dot{\mathcal F}(\mathbf u^n)
+\frac{\Delta t^2}{3}\dot{\mathcal F}(\mathbf u^*) \\
&\quad +\Delta t\,\mathcal G(\mathbf u^{n+1})
-\frac{\Delta t^2}{6}\dot{\mathcal G}(\mathbf u^{n+1})
-\frac{\Delta t^2}{3}\dot{\mathcal G}(\mathbf u^*),
\end{aligned}
\label{eq:num_update}
\end{equation*}
from \eqref{eq:combined_exact}. Since \(\mathbf u^n=\mathbf u(t_n)\), the \(\mathcal F\)-terms cancel exactly, and the \(\mathbf U^*\)-terms also cancel exactly because they appear identically in both expressions. Thus we obtain
\begin{equation}
\mathbf e^{n+1}
=
\Delta t\Bigl[\mathcal G(\mathbf u^{n+1})-\mathcal G(\mathbf u(t_{n+1}))\Bigr]
-\frac{\Delta t^2}{6}
\Bigl[\dot{\mathcal G}(\mathbf u^{n+1})-\dot{\mathcal G}(\mathbf u(t_{n+1}))\Bigr]
+\mathcal O(\Delta t^5).
\label{eq:error_final}
\end{equation}
The remaining \(\mathcal O(\Delta t^5)\) term comes solely from the quadrature errors in \eqref{eq:F_int_star} and \eqref{eq:G_int_star}; no further expansion of \(\mathbf u^*\) is needed.

Since \(\mathcal G\) and \(\dot{\mathcal G}\) are smooth, they are Lipschitz continuous on compact sets. Let \(L_G\) and \(L_{\dot G}\) denote their Lipschitz constants. Taking norms in \eqref{eq:error_final} gives
\begin{equation*}
\|\mathbf e^{n+1}\|
\le
\Delta t\,L_G\,\|\mathbf e^{n+1}\|
+ \frac{\Delta t^2}{6}L_{\dot G}\,\|\mathbf e^{n+1}\|
+ C\Delta t^5,
\label{eq:norm_ineq}
\end{equation*}
for some constant \(C>0\) independent of \(\Delta t\). Rearranging,
\begin{equation*}
\left(1-\Delta t L_G-\frac{\Delta t^2}{6}L_{\dot G}\right)\|\mathbf e^{n+1}\| \le C\Delta t^5.
\label{eq:final_ineq}
\end{equation*}
For sufficiently small \(\Delta t\), the prefactor is positive and bounded away from zero. Hence there exists \(C'>0\) such that
\begin{equation*}
\|\mathbf e^{n+1}\| \le C' \Delta t^5.
\label{eq:local_error_bound}
\end{equation*}
Thus the local truncation error is \(\mathcal O(\Delta t^5)\).

Finally, over a fixed time interval \([0,T]\) with \(N=T/\Delta t\) steps, the global error satisfies
\begin{equation*}
\|\mathbf u^N-\mathbf u(T)\|
\le
\sum_{j=0}^{N-1} \|\mathbf e^{j+1}\|
\le
N\cdot C' \Delta t^5
=
\frac{T}{\Delta t} C' \Delta t^5
=
\mathcal O(\Delta t^4).
\label{eq:global_error}
\end{equation*}
Therefore the method is fourth-order accurate in time.
\end{proof}

\section{Linear stability analysis}
\label{sec:stability}

To rigorously assess the robustness of the split two-derivative scheme, we investigate its stability properties on the standard scalar test equation
\begin{equation}
    u_t = \lambda_E u + \lambda_I u, \qquad \lambda_E,\lambda_I\in\mathbb{C},
\end{equation}
where the explicit and implicit split components are identified with $\mathcal{F}=\lambda_E u$ and $\mathcal{G}=\lambda_I u$, respectively. We define the complex parameters
\[
    z_E = \lambda_E \Delta t, \qquad z_I = \lambda_I \Delta t,
\]
and analyze the stability function $R(z_E,z_I)$ for various limiting regimes. In the following, we set $\Re(z)\le 0$ for the implicit component to correspond to the stable, stiff region.

\subsection{Pure implicit reduction ($\mathcal{F}=0$)}

When the explicit component vanishes, the method reduces to a fully implicit two-derivative scheme. Applying the method to $u_t=\lambda_I u$ yields the stability function
\begin{equation}
    R_I(z_I)
    =
    \frac{-2(5z_I^2+12z_I-24)}{(z_I^2-6z_I+6)(z_I^2-4z_I+8)}.
    \label{eq:pure_implicit_stability}
\end{equation}
The denominator factors as
\[
    (z_I^2-6z_I+6)(z_I^2-4z_I+8)
    =
    \prod_{j=1}^4 (z_I - p_j),
\]
with poles
\[
    p_1 = 3-\sqrt{3},\quad p_2 = 3+\sqrt{3},\quad p_3 = 2+2\mathrm{i},\quad p_4 = 2-2\mathrm{i}.
\]
Since $\Re(p_j)>0$ for all $j$, $R_I(z_I)$ is analytic in the left half-plane $\Re(z_I)\le 0$. Moreover, direct evaluation gives
\[
    R_I(0)=1,\qquad \lim_{|z_I|\to\infty} R_I(z_I) = 0,
\]
so the implicit substructure is both A-stable and L-stable. In particular, it provides strong damping of highly oscillatory or stiff implicit modes.

\subsection{Pure explicit reduction ($\mathcal{G}=0$)}

In the absence of implicit terms, the two-derivative construction collapses to the standard fourth-order Taylor polynomial. The corresponding explicit stability function is
\begin{equation*}
    R_E(z_E) = 1 + z_E + \frac{1}{2}z_E^2 + \frac{1}{6}z_E^3 + \frac{1}{24}z_E^4,
\end{equation*}
which is the stability polynomial of the classical explicit fourth-order Runge--Kutta method. The stability region $\mathcal{S}_E = \{z_E\in\mathbb{C}: |R_E(z_E)|\le 1\}$ is bounded. Along the negative real axis, the boundary is attained at $r_E \approx -2.785$, i.e.
\[
    |R_E(-r_E)| = 1, \qquad r_E \in \mathbb{R}^+.
\]
Consequently, the explicit part imposes a finite Courant--Friedrichs--Lewy (CFL) restriction of the form $\Delta t \lesssim r_E / |\lambda_E|$ when the implicit part is inactive.

\subsection{Fully split case ($\mathcal{F}\neq 0,\ \mathcal{G}\neq 0$)}

For the general split test equation, let $s = z_E + z_I$. The stage amplification factor is defined by
\begin{equation*}
    Q(z_E,z_I)
    =
    \frac{1+\frac{1}{2}z_E+\frac{1}{8}z_Es}{1-\frac{1}{2}z_I+\frac{1}{8}z_Is},
\end{equation*}
and the one-step stability function reads
\begin{equation}
    R(z_E,z_I)
    =
    \frac{1+z_E+\frac{1}{6}z_Es+\frac{1}{3}(z_E-z_I)sQ(z_E,z_I)}{1-z_I+\frac{1}{6}z_Is}.
    \label{eq:full_stability}
\end{equation}
This bivariate rational function governs the linear stability of the coupled explicit--implicit scheme.


To further illustrate the coupled stability behavior, Fig.~\ref{fig:real_axis_stability_slice} presents a real-axis stability slice of the bivariate stability function \(R(z_E,z_I)\). In this plot, both stability variables are restricted to the non-positive real axis, namely \(z_E\le 0\) and \(z_I\le 0\). The shaded region represents the set of real pairs satisfying \(|R(z_E,z_I)|\le 1\). The figure shows that the method retains the explicit stability restriction when the implicit component is absent, while it exhibits stronger damping as the implicit component moves into the strongly stiff negative real regime. In particular, for a fixed explicit component \(z_E\), the stability factor decreases as \(|z_I|\) becomes large, which is consistent with the asymptotic behavior \(R(z_E,z_I)\sim -z_E/z_I\) as \(|z_I|\to\infty\). Therefore, this plot supports the stiff-decay property in the pure implicit limit and in the strong implicit-stiffness limit with a fixed explicit component. It should not be interpreted as a complete characterization of the full complex two-variable stability region.

\begin{figure}[htbp]
    \centering
\includegraphics[width=0.72\textwidth]{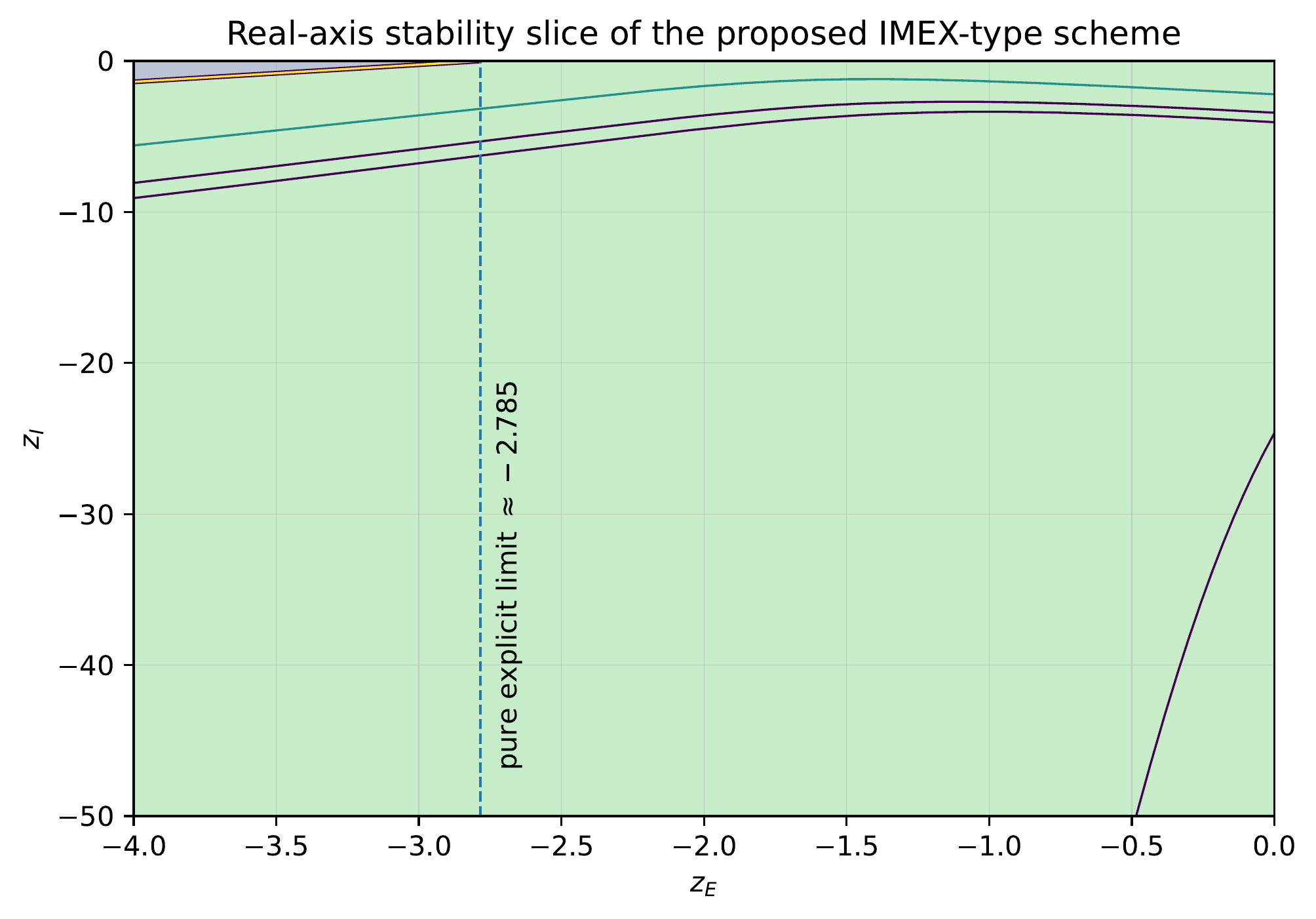}

\caption{Real-axis stability slice of the proposed IMEX-type two-derivative scheme. The shaded region denotes the set of non-positive real pairs \((z_E,z_I)\) satisfying \(|R(z_E,z_I)|\le 1\), and the solid contour represents the boundary \(|R(z_E,z_I)|=1\). The dashed vertical line indicates the pure explicit stability limit on the negative real axis. The plot illustrates that the method retains the explicit stability restriction while exhibiting stiff decay as the implicit component becomes increasingly stiff. This figure is only a real-axis slice and is not intended as a complete characterization of the full complex two-variable stability region.}
    \label{fig:real_axis_stability_slice}
\end{figure}

To characterize the interplay between the two split components, we examine the asymptotic behavior of $R(z_E,z_I)$ in the stiff limit.

\begin{enumerate}
    \item \textit{Fixed explicit part, stiff implicit part.} 
    For any fixed $z_E\in\mathbb{C}$ and $|z_I|\to\infty$ with $\Re(z_I)\le 0$, we have $s\sim z_I$. A direct asymptotic expansion of (\ref{eq:full_stability}) yields
    \[
        Q(z_E,z_I) \sim \frac{z_E}{z_I}, \qquad 
        R(z_E,z_I) \sim -\frac{z_E}{z_I} \to 0.
    \]
    Thus, the scheme remains L-stable with respect to the implicit component even in the presence of a bounded explicit contribution. The implicit part completely dampens the solution when the implicit time step tends to infinity.

    \item \textit{Simultaneously stiff explicit and implicit parts.} 
    Let both parameters grow unbounded with a fixed ratio, i.e. $z_E = c\, z_I$ and $|z_I|\to\infty$, where $c\in\mathbb{C}$ is a constant. Substituting $s=(1+c)z_I$ into (\ref{eq:full_stability}) and retaining the leading-order terms gives
    \begin{equation*}
        \lim_{|z_I|\to\infty} R(c\,z_I,z_I)
        =
        \frac{c(3c-1)}{1+c},
        \label{eq:double_stiff_limit}
    \end{equation*}
    provided $1+c\neq 0$. This asymptotic limit is generally nonzero. Consequently, stability in the doubly stiff regime requires
    \begin{equation*}
        \left|\frac{c(3c-1)}{1+c}\right| \le 1,
        \label{eq:double_stiff_condition}
    \end{equation*}
    which imposes a constraint on the relative magnitude of the explicit and implicit Courant numbers. This condition is nontrivial and must be respected when both components are evolved with large step sizes (e.g., in advection--diffusion problems where the explicit advection CFL number and the implicit diffusion number are both large).
\end{enumerate}

\section{Comparison with Classical Multi-Stage Fourth-Order IMEX--RK Schemes}
\label{sec:discussion_comparison}

The proposed method is not a classical one-derivative IMEX Runge--Kutta method. Classical multi-stage fourth-order IMEX--RK schemes achieve fourth-order accuracy by introducing several intermediate stages and by imposing explicit, implicit, and mixed coupling order conditions on the Butcher coefficients. In contrast, the present method is a two-derivative IMEX formulation. It uses not only the split vector fields \(\mathcal F\) and \(\mathcal G\), but also their temporal derivatives evaluated along the full vector field \(\mathcal F+\mathcal G\). This structural difference leads to three main advantages: mixed consistency at the derivative level, a compact two-stage fourth-order construction, and stronger stiff-mode damping in the strongly stiff regime.

\subsection{Mixed consistency at the derivative level}
\label{subsec:mixed_consistency}

Consider the additively split system \eqref{eq:split_system}
The exact second time derivative is
\begin{equation*}
\begin{aligned}
\mathbf u_{tt}
&=
\frac{d}{dt}
\left[
\mathcal F(\mathbf u)+\mathcal G(\mathbf u)
\right]  \\                                                 \
&=
\left(
\mathcal F_{\mathbf u}
+
\mathcal G_{\mathbf u}
\right)
\left(
\mathcal F+\mathcal G
\right)\\                                                   \
&=
\mathcal F_{\mathbf u}\mathcal F
+
\mathcal F_{\mathbf u}\mathcal G
+
\mathcal G_{\mathbf u}\mathcal F
+
\mathcal G_{\mathbf u}\mathcal G .
\end{aligned}
\label{eq:second_derivative_split}
\end{equation*}
Hence, the Taylor expansion of the exact solution contains not only the self-interaction terms
\(
\mathcal F_{\mathbf u}\mathcal F\;\text{and}\;
\mathcal G_{\mathbf u}\mathcal G,
\)
but also the mixed explicit--implicit interaction terms
\(
\mathcal F_{\mathbf u}\mathcal G
\;\text{and}\;
\mathcal G_{\mathbf u}\mathcal F.
\)
These mixed terms are essential when the explicit and implicit split components are coupled or non-commuting.

The proposed method incorporates these mixed interactions directly through the full directional temporal derivatives
\eqref{eq:Fdot} and \eqref{eq:Gdot}.
Expanding the two derivatives gives
\begin{equation*}
\dot{\mathcal F}
=
\mathcal F_{\mathbf u}\mathcal F
+
\mathcal F_{\mathbf u}\mathcal G,
\label{eq:expanded_F_derivative}
\end{equation*}
and
\begin{equation*}
\dot{\mathcal G}
=
\mathcal G_{\mathbf u}\mathcal F
+
\mathcal G_{\mathbf u}\mathcal G.
\label{eq:expanded_G_derivative}
\end{equation*}
Therefore, the mixed terms
\(
\mathcal F_{\mathbf u}\mathcal G
\;\text{and}\;
\mathcal G_{\mathbf u}\mathcal F
\)
are included analytically in the derivative evaluation itself. This means that mixed consistency is built into the formulation at the derivative level, rather than being imposed only through global algebraic order conditions.

This point can be seen more clearly for a linear split system
\begin{equation*}
\mathcal F(\mathbf u)=A\mathbf u,
\qquad
\mathcal G(\mathbf u)=B\mathbf u.
\label{eq:linear_split_comparison}
\end{equation*}
Then
\begin{equation*}
\dot{\mathcal F}
=
A(A+B)\mathbf u
=
A^2\mathbf u+AB\mathbf u,
\label{eq:Fdot_linear_split}
\end{equation*}
and
\begin{equation*}
\dot{\mathcal G}
=
B(A+B)\mathbf u
=
BA\mathbf u+B^2\mathbf u.
\label{eq:Gdot_linear_split}
\end{equation*}
Thus, the proposed two-derivative formulation explicitly contains both mixed matrix products
\(
AB\mathbf u
\;\text{and}\;
BA\mathbf u.
\)
If \(AB\neq BA\), these two mixed contributions are different and cannot be merged into a single commutative term. The exact second-order expansion satisfies
\begin{equation*}
(A+B)^2
=
A^2+AB+BA+B^2.
\label{eq:noncommuting_expansion}
\end{equation*}
Therefore, both \(AB\) and \(BA\) must be represented correctly in order to preserve high-order accuracy for non-commuting split systems.

Classical multi-stage IMEX--RK methods recover these mixed effects in a different way. A classical \(s\)-stage IMEX--RK method has the form
\eqref{eq:imex_rk_stage} and \eqref{eq:imex_rk_update}.
For fourth-order accuracy, classical multi-stage IMEX--RK methods must satisfy not only the usual explicit and implicit Runge--Kutta order conditions, but also mixed coupling conditions involving both Butcher tableaux, for example
\begin{equation*}
b^T A^E A^I c=1/24
\quad \text{and}\quad
b^T A^I A^E c=1/24
\end{equation*}
These algebraic coupling conditions are necessary to reproduce the correct mixed Taylor terms at the global truncation-error level. However, the mixed interactions are enforced indirectly through the coefficients of the method. In contrast, the proposed scheme incorporates the mixed terms directly through the full directional derivatives in \eqref{eq:Fdot} and \eqref{eq:Gdot}. This provides a more transparent mechanism for mixed consistency and explains why the proposed method maintains fourth-order accuracy in the non-commuting split tests.

\subsection{Compact stage structure and implicit-solve efficiency}
\label{subsec:compact_efficiency}

The proposed method achieves fourth-order accuracy through Hermite-type two-derivative quadrature rather than through many Runge--Kutta stages. For a sufficiently smooth function \(q(t)\), the following one-sided Hermite formula holds:
\begin{equation*}
\int_{t_n}^{t_{n+1}} q(t)\,dt
=
\Delta t\,q(t_n)
+
\frac{\Delta t^2}{6}
\left[
q'(t_n)
+
2q'\!\left(t_n+\frac{\Delta t}{2}\right)
\right]
+
\mathcal O(\Delta t^5).
\label{eq:left_hermite_formula_comparison}
\end{equation*}
Similarly, using the right endpoint gives
\begin{equation*}
\int_{t_n}^{t_{n+1}} q(t)\,dt
=
\Delta t\,q(t_{n+1})
-
\frac{\Delta t^2}{6}
\left[
q'(t_{n+1})
+
2q'\!\left(t_n+\frac{\Delta t}{2}\right)
\right]
+
\mathcal O(\Delta t^5).
\label{eq:right_hermite_formula_comparison}
\end{equation*}
The proposed scheme applies the left-endpoint formula to the explicit component and the right-endpoint formula to the implicit component. The intermediate value \(\mathbf u^*\) approximates the midpoint solution:
\begin{equation*}
\mathbf u^*
=
\mathbf u\!\left(t_n+\frac{\Delta t}{2}\right)
+
\mathcal O(\Delta t^3).
\label{eq:midpoint_approximation_comparison}
\end{equation*}
Since the midpoint derivative terms in the final update are multiplied by \(\Delta t^2\), this midpoint accuracy is sufficient to preserve a local truncation error of order
\begin{equation*}
\mathcal O(\Delta t^5),
\end{equation*}
and hence fourth-order global accuracy.

Therefore, the proposed method obtains fourth-order accuracy with only one intermediate stage and two implicit solves per time step. This is more compact than classical multi-stage fourth-order IMEX--RK schemes, which usually require several stages to satisfy all explicit, implicit, and mixed coupling order conditions.

The lower number of implicit stages also leads to a natural efficiency advantage when the comparison is normalized by the number of implicit solves. Let \(M\) be a fixed implicit-solve budget. Since the proposed method uses two implicit solves per time step, it can take approximately
\begin{equation*}
N_{\rm Pro}
\approx
\frac{M}{2}
\label{eq:nsteps_pro_comparison}
\end{equation*}
time steps. If the KC--ARK4 reference method requires \(s_I\) nontrivial implicit solves per time step, then it can take approximately
\begin{equation*}
N_{\rm IMEX}
\approx
\frac{M}{s_I}
\label{eq:nsteps_imex_comparison}
\end{equation*}
time steps. For fourth-order methods, the leading global errors behave as
\begin{equation*}
E_{\rm Pro}
\approx
C_{\rm Pro}
\left(
\frac{T}{N_{\rm Pro}}
\right)^4,
\qquad
E_{\rm IMEX}
\approx
C_{\rm IMEX}
\left(
\frac{T}{N_{\rm IMEX}}
\right)^4.
\label{eq:error_budget_comparison}
\end{equation*}
Thus,
\begin{equation*}
\frac{E_{\rm IMEX}}{E_{\rm Pro}}
\approx
\frac{C_{\rm IMEX}}{C_{\rm Pro}}
\left(
\frac{s_I}{2}
\right)^4.
\label{eq:error_ratio_budget_comparison}
\end{equation*}
For the KC--ARK4 reference method used in the numerical experiments, \(s_I=5\). Therefore, the stage-count factor alone is
\begin{equation*}
\left(
\frac{5}{2}
\right)^4
=
39.0625.
\label{eq:stage_count_factor_comparison}
\end{equation*}
This shows why, under the same implicit-solve budget, the proposed two-stage method can achieve higher accuracy per implicit solve, especially in stiff regimes where the stronger damping of fast components also contributes to the error reduction.

\subsection{Stronger stiff-mode damping}
\label{subsec:stiff_damping}


Consider the purely implicit scalar test equation
\begin{equation*}
{\bf u}_t=\lambda {\bf u},
\qquad
z=\lambda\Delta t,
\qquad
\operatorname{Re}(z)<0.
\label{eq:scalar_stiff_test_comparison}
\end{equation*}
In this case,
\begin{equation*}
\mathcal F=0,
\qquad
\mathcal G({\bf u})=\lambda {\bf u},
\qquad
\dot{\mathcal G}({\bf u})=\lambda^2{\bf u}.
\end{equation*}
The stage equation becomes
\begin{equation*}
{\bf u}^*
=
{\bf u}^n
+
\frac{z}{2}{\bf u}^*
-
\frac{z^2}{8}{\bf u}^*.
\label{eq:scalar_stage_comparison}
\end{equation*}
Hence,
\begin{equation}
{\bf u}^*
=
\frac{1}
{
1-\frac{z}{2}+\frac{z^2}{8}
}
{\bf u}^n.
\label{eq:scalar_stage_solution_comparison}
\end{equation}
The final update becomes
\begin{equation*}
{\bf u}^{n+1}
=
{\bf u}^n
+
z {\bf u}^{n+1}
-
\frac{z^2}{6}
\left(
{\bf u}^{n+1}+2{\bf u}^*
\right).
\label{eq:scalar_final_update_comparison}
\end{equation*}
Equivalently,
\begin{equation}
\left(
1-z+\frac{z^2}{6}
\right)
{\bf u}^{n+1}
=
{\bf u}^n
-
\frac{z^2}{3}{\bf u}^*.
\label{eq:scalar_final_update_rearranged_comparison}
\end{equation}
Substituting \eqref{eq:scalar_stage_solution_comparison} into \eqref{eq:scalar_final_update_rearranged_comparison} gives the stability function
\begin{equation*}
R_{\rm Pro}(z)
=
\frac{
1-\frac{z}{2}-\frac{5z^2}{24}
}{
\left(
1-\frac{z}{2}+\frac{z^2}{8}
\right)
\left(
1-z+\frac{z^2}{6}
\right)
}.
\label{eq:R_pro_comparison}
\end{equation*}
Therefore, as \(z\to-\infty\),
\begin{equation*}
R_{\rm Pro}(z)
=
-\frac{10}{z^2}
+
\mathcal O(z^{-3}).
\label{eq:R_pro_asymptotic_comparison}
\end{equation*}
This proves that the proposed method is L-stable and, more importantly, that its stiff-mode decay is of order
\begin{equation*}
R_{\rm Pro}(z)=\mathcal O(z^{-2}),
\qquad
z\to-\infty.
\label{eq:R_pro_order_comparison}
\end{equation*}

This asymptotic behavior of the amplification factor \(R_{\rm IMEX}(z)\) of the IMEX--RK scheme follows by expanding the rational stability function
\begin{equation*}
R_{\rm IMEX}(z)=1+z\,b^T(I-zA_I)^{-1}(I+zA_E)\mathbf{1}
\end{equation*}
for large \(|z|\). Using the Neumann-type expansion
\begin{equation*}
(I-zA_I)^{-1}=-z^{-1}A_I^{-1}-z^{-2}A_I^{-2}+\mathcal O(z^{-3}),
\end{equation*}
and invoking the RK order conditions to eliminate the constant term, the leading-order term reduces to \(\alpha/z\), with
\begin{equation*}
\alpha=-b^T A_I^{-1}(I+A_E)\mathbf{1}\neq 0
\end{equation*}
for the specific fourth-order IMEX scheme under consideration. 
For the KC--ARK4 reference method employed in the numerical experiments, the implicit stability factor therefore exhibits the large-\(|z|\) asymptotics
\begin{equation*}
R_{\rm IMEX}(z)=\frac{\alpha}{z}+\mathcal O(z^{-2}),\qquad z\to-\infty,
\label{eq:R_imex_asymptotic_comparison}
\end{equation*}
where \(\alpha\neq 0\).

Consequently, comparing the decay rates gives
\begin{equation*}
R_{\rm IMEX}(z)=\mathcal O(z^{-1}),\qquad R_{\rm Pro}(z)=\mathcal O(z^{-2}).
\label{eq:decay_order_comparison}
\end{equation*}
The relative damping advantage satisfies
\begin{equation*}
\frac{|R_{\rm IMEX}(z)|}{|R_{\rm Pro}(z)|}
=
\mathcal O(|z|),
\qquad
z\to-\infty.
\label{eq:relative_damping_comparison}
\end{equation*}
Thus, as the stiffness increases, the proposed method damps fast stiff modes more strongly than the KC--ARK4 reference method.

This theoretical damping mechanism explains the numerical observations in the scaled stiff split system and the high-frequency heat equation. When the stiff eigenvalues satisfy
\begin{equation*}
|\lambda|\Delta t\gg 1,
\end{equation*}
the proposed method produces smaller fast-mode residuals and smaller errors under the same implicit-solve budget.

In summary, the proposed method does not claim a uniformly smaller error constant for all problems. Its main advantage lies in the combination of mixed consistency, fourth-order accuracy with a compact two-stage structure, stronger stiff-mode damping in the above limiting regimes, and higher accuracy per implicit solve for stiff multiscale systems.


\section{Numerical validation}
\label{sec:numerical_validation}

This section is redesigned to verify the specific advantages predicted by the comparison in Section~\ref{sec:discussion_comparison}, rather than to claim that the proposed method is uniformly more accurate than classical multi-stage fourth-order IMEX--RK schemes for all split problems. The numerical tests focus on three mechanisms: mixed consistency for non-commuting split systems, enhanced damping of stiff modes, and improved accuracy under the same implicit-solve budget. Consequently, comparisons based only on the same number of time steps are not used as the primary efficiency metric, since different IMEX schemes may require very different numbers of implicit stage solves per step.

The proposed method is denoted by Pro. The reference method is the Kennedy--Carpenter ARK4(3)6L[2]SA additive Runge--Kutta method \cite{KennedyCarpenter2003}, denoted by KC--ARK4. This method is used as a representative classical multi-stage fourth-order one-derivative IMEX--RK scheme. Its Butcher tableaux are given in Section~\ref{subsec:kc_ark4_reference}. KC--ARK4 has six stages and five nontrivial implicit solves per time step, because its first implicit stage has zero diagonal coefficient whereas stages 2--6 have nonzero diagonal coefficient. In contrast, the proposed method uses only two implicit solves per time step. The stage-cost comparison is summarized in Table~\ref{tab:stage_cost}.

\begin{table}[htbp]
\centering
\caption{Stage-cost comparison between the proposed method and the KC--ARK4 reference method.}
\label{tab:stage_cost}
\begin{tabular}{lcccc}
\toprule
Method & Order & Stages & Implicit solves/step & Steps per 100 implicit solves \\
\midrule
Pro       & 4 & 2 & 2 & 50 \\
KC--ARK4  & 4 & 6 & 5 & 20 \\
\bottomrule
\end{tabular}
\end{table}

\subsection{Mixed consistency for a non-commuting split system}
\label{subsec:mixed_consistency_test}

We first consider the linear split system
\begin{equation}
    \frac{d\mathbf u}{dt}=A\mathbf u+B_\kappa \mathbf u,
    \qquad
    A=\begin{pmatrix}0&2\\-1&0\end{pmatrix},
    \qquad
    B_\kappa=\kappa\begin{pmatrix}-3&1\\0.5&-2\end{pmatrix}.
\end{equation}
The matrices satisfy $AB_\kappa\neq B_\kappa A$, and therefore the problem directly tests whether the numerical method correctly retains the mixed split terms. The exact solution is computed by
\begin{equation}
    \mathbf u(t)=\exp\bigl((A+B_\kappa)t\bigr)\mathbf u_0 .
\end{equation}
We take $\kappa=100$, $\mathbf u_0=(1,-0.5)^T$, and $T=0.02$.

To demonstrate that evaluating the derivatives along the full vector field is essential, we compare the proposed full method with an intentionally incomplete variant. The full method uses
\begin{equation}
    \dot F=A(A+B_\kappa)\mathbf u=A^2\mathbf u+AB_\kappa\mathbf u,
    \qquad
    \dot G=B_\kappa(A+B_\kappa)\mathbf u=B_\kappa A\mathbf u+B_\kappa^2\mathbf u,
\end{equation}
whereas the incomplete variant keeps only the self-interaction terms $A^2\mathbf u$ and $B_\kappa^2\mathbf u$ and omits the mixed terms $AB_\kappa\mathbf u$ and $B_\kappa A\mathbf u$.

\begin{table}[htbp]
\centering
\caption{Mixed-consistency test for the non-commuting split system with $\kappa=100$.}
\label{tab:mixed_consistency}
\begin{tabular}{lccc}
\toprule
Method & $N$ & Error & Order \\
\midrule
Pro-full & 10 & $5.16\times 10^{-6}$ & -- \\
Pro-full & 20 & $4.10\times 10^{-7}$ & 3.66 \\
Pro-full & 40 & $2.94\times 10^{-8}$ & 3.80 \\
Pro-full & 80 & $1.98\times 10^{-9}$ & 3.89 \\
Pro-full & 160 & $1.29\times 10^{-10}$ & 3.94 \\
Pro-full & 320 & $8.20\times 10^{-12}$ & 3.97 \\
Pro-incomplete & 10 & $4.50\times 10^{-5}$ & -- \\
Pro-incomplete & 20 & $2.24\times 10^{-5}$ & 1.01 \\
Pro-incomplete & 40 & $1.09\times 10^{-5}$ & 1.03 \\
Pro-incomplete & 80 & $5.40\times 10^{-6}$ & 1.02 \\
Pro-incomplete & 160 & $2.68\times 10^{-6}$ & 1.01 \\
Pro-incomplete & 320 & $1.34\times 10^{-6}$ & 1.01 \\
\bottomrule
\end{tabular}
\end{table}

\begin{figure}[htbp]
\centering
\includegraphics[width=0.72\textwidth]{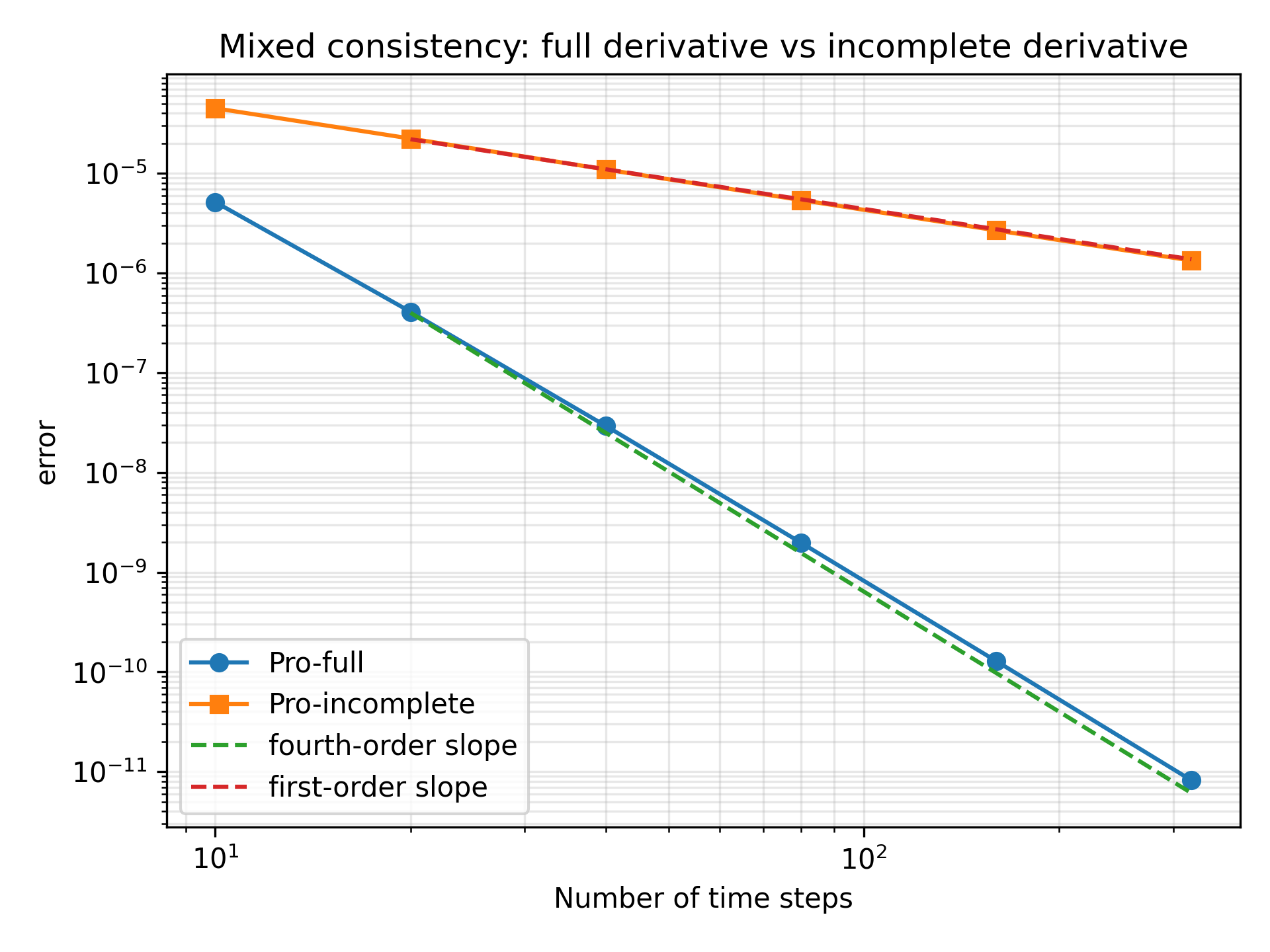}
\caption{Mixed-consistency test. The full-vector-field derivative retains fourth-order convergence, while the incomplete derivative variant loses high-order accuracy.}
\label{fig:mixed_consistency}
\end{figure}

Table~\ref{tab:mixed_consistency} and Figure~\ref{fig:mixed_consistency} show that Pro-full achieves fourth-order convergence, whereas Pro-incomplete is reduced to approximately first order. This confirms that the mixed derivative terms are not merely a formal reformulation; they are essential for preserving high-order accuracy when the split components are non-commuting.

\subsection{Increasing stiffness in the non-commuting split system}
\label{subsec:increasing_kappa}

We next investigate whether the advantage of the proposed method becomes more pronounced as the implicit stiffness increases. The same non-commuting split system is used, but the stiffness parameter $\kappa$ is varied. The final time is set to $T=0.002$, and the comparison is performed under a fixed budget of 20 implicit solves. Thus Pro takes $N_{\rm Pro}=10$ time steps, whereas KC--ARK4 takes $N_{\rm KC}=4$ time steps.

For each $\kappa$, the stiffness index is measured by $\rho(B_\kappa)\Delta t_{\rm KC}$, where $\rho(B_\kappa)$ denotes the spectral radius of the implicit matrix and $\Delta t_{\rm KC}=T/N_{\rm KC}$. The error ratio
\begin{equation}
    \frac{\text{Error(KC--ARK4)}}{\text{Error(Pro)}}
\end{equation}
is used to measure the relative advantage of the proposed method.

\begin{table}[htbp]
\centering
\caption{Increasing stiffness in the non-commuting split system under the same budget of 20 implicit solves.}
\label{tab:increasing_kappa}
\begin{tabular}{ccccc}
\toprule
$\kappa$ & $\rho(B_\kappa)\Delta t_{\rm KC}$ & Error Pro & Error KC--ARK4 & Ratio \\
\midrule
100 & 0.17 & $3.95\times 10^{-8}$ & $2.52\times 10^{-7}$ & $6.39\times 10^{0}$ \\
200 & 0.34 & $6.07\times 10^{-7}$ & $4.15\times 10^{-6}$ & $6.83\times 10^{0}$ \\
500 & 0.84 & $6.63\times 10^{-6}$ & $5.54\times 10^{-5}$ & $8.36\times 10^{0}$ \\
1000 & 1.68 & $5.40\times 10^{-6}$ & $6.67\times 10^{-5}$ & $1.24\times 10^{1}$ \\
2000 & 3.37 & $6.78\times 10^{-7}$ & $1.08\times 10^{-5}$ & $1.59\times 10^{1}$ \\
5000 & 8.42 & $1.86\times 10^{-9}$ & $2.02\times 10^{-4}$ & $1.09\times 10^{5}$ \\
\bottomrule
\end{tabular}
\end{table}

\begin{figure}[htbp]
\centering
\includegraphics[width=0.72\textwidth]{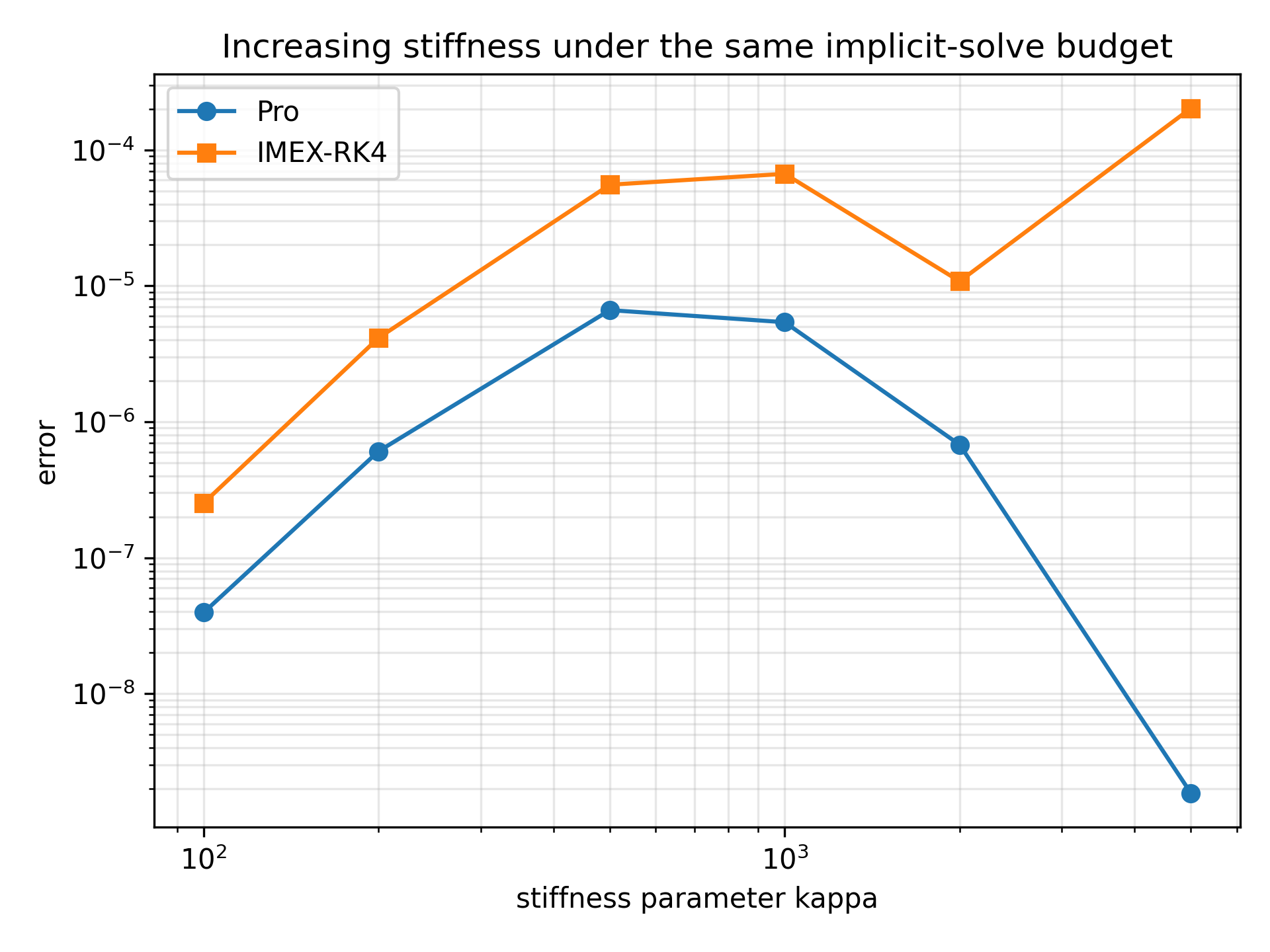}
\caption{Error comparison for increasing stiffness under the same implicit-solve budget.}
\label{fig:increasing_kappa_errors}
\end{figure}

\begin{figure}[htbp]
\centering
\includegraphics[width=0.72\textwidth]{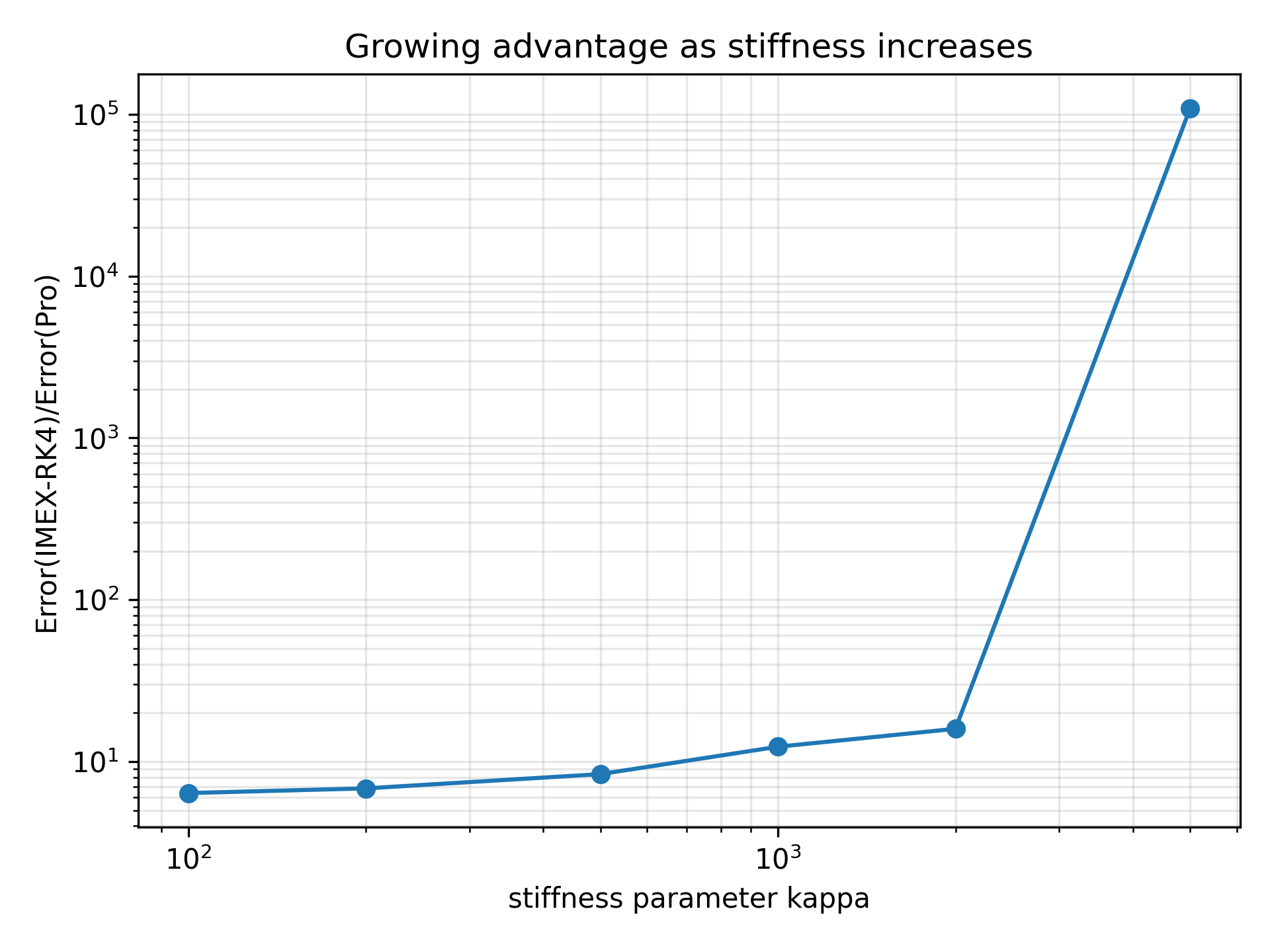}
\caption{Error ratio as the stiffness parameter increases. The growing ratio indicates that the proposed method becomes more favorable in stronger stiff regimes.}
\label{fig:increasing_kappa_ratio}
\end{figure}

The results in Table~\ref{tab:increasing_kappa} and Figures~\ref{fig:increasing_kappa_errors}--\ref{fig:increasing_kappa_ratio} show that the relative advantage of Pro generally increases with the stiffness level. This behavior is consistent with the theoretical stiff-decay analysis: in the strongly stiff regime, the proposed method damps fast implicit modes more strongly than the KC--ARK4 reference method.

\subsection{Asymptotic stiff-mode damping}
\label{subsec:scalar_damping}

We now isolate the stiff-mode damping mechanism by considering the purely implicit scalar equation
\begin{equation}
    u'=\lambda u,
    \qquad z=\lambda\Delta t<0.
\end{equation}
For this test, the proposed method satisfies
\begin{equation}
    R_{\rm Pro}(z)=O(z^{-2}),\qquad z\to -\infty,
\end{equation}
whereas the KC--ARK4 reference method used here has the slower decay
\begin{equation}
    R_{\rm IMEX}(z)=O(z^{-1}),\qquad z\to -\infty.
\end{equation}

\begin{table}[htbp]
\centering
\caption{Asymptotic damping factors for the scalar stiff mode.}
\label{tab:scalar_damping}
\begin{tabular}{cccccc}
\toprule
$z$ & $|R_{\rm Pro}|$ & $|z|^2|R_{\rm Pro}|$ & $|R_{\rm IMEX}|$ & $|z||R_{\rm IMEX}|$ & Ratio \\
\midrule
$-10^{2}$ & $8.84\times 10^{-4}$ & 8.84 & $7.57\times 10^{-2}$ & 7.57 & $8.57\times 10^{1}$ \\
$-10^{3}$ & $9.88\times 10^{-6}$ & 9.88 & $9.14\times 10^{-3}$ & 9.14 & $9.25\times 10^{2}$ \\
$-10^{4}$ & $9.99\times 10^{-8}$ & 9.99 & $9.31\times 10^{-4}$ & 9.31 & $9.33\times 10^{3}$ \\
$-10^{5}$ & $1.00\times 10^{-9}$ & 10.00 & $9.33\times 10^{-5}$ & 9.33 & $9.33\times 10^{4}$ \\
$-10^{6}$ & $1.00\times 10^{-11}$ & 10.00 & $9.33\times 10^{-6}$ & 9.33 & $9.33\times 10^{5}$ \\
\bottomrule
\end{tabular}
\end{table}

\begin{figure}[htbp]
\centering
\includegraphics[width=0.72\textwidth]{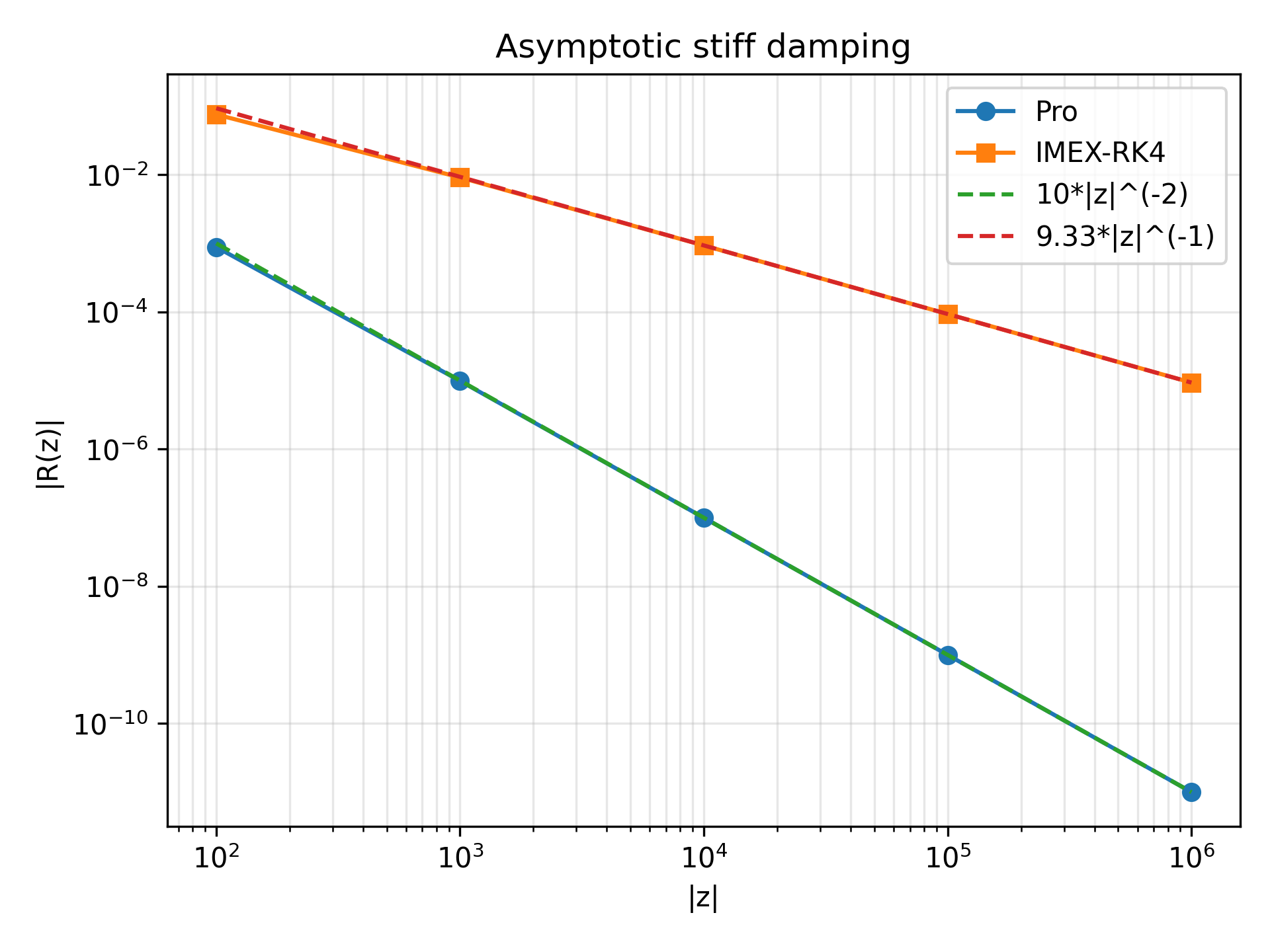}
\caption{Asymptotic stiff damping on the negative real axis. Pro follows the $O(|z|^{-2})$ reference slope, whereas KC--ARK4 follows the slower $O(|z|^{-1})$ slope.}
\label{fig:scalar_damping}
\end{figure}

The nearly constant values of $|z|^2|R_{\rm Pro}|$ and $|z||R_{\rm IMEX}|$ verify the predicted asymptotic decay rates. Therefore, the relative damping advantage grows approximately linearly with $|z|$.

\subsection{Work-normalized comparison for a strongly stiff split system}
\label{subsec:work_normalized}

To examine the efficiency advantage of the compact two-stage structure, we fix $\kappa=1000$ in the non-commuting split system and vary the total number of implicit solves. Since Pro uses two implicit solves per step while KC--ARK4 uses five nontrivial implicit solves per step, Pro can take more time steps under the same implicit-solve budget.

\begin{table}[htbp]
\centering
\caption{Work-normalized comparison for the scaled split system with $\kappa=1000$.}
\label{tab:work_normalized}
\begin{tabular}{cccccc}
\toprule
Budget & $N_{\rm Pro}$ & $N_{\rm KC}$ & Error Pro & Error KC--ARK4 & Ratio \\
\midrule
20 & 10 & 4 & $5.40\times 10^{-6}$ & $6.67\times 10^{-5}$ & 12.35 \\
30 & 15 & 6 & $1.25\times 10^{-6}$ & $1.19\times 10^{-5}$ & 9.54 \\
40 & 20 & 8 & $4.29\times 10^{-7}$ & $3.63\times 10^{-6}$ & 8.46 \\
60 & 30 & 12 & $9.29\times 10^{-8}$ & $6.99\times 10^{-7}$ & 7.53 \\
80 & 40 & 16 & $3.08\times 10^{-8}$ & $2.19\times 10^{-7}$ & 7.10 \\
120 & 60 & 24 & $6.40\times 10^{-9}$ & $4.29\times 10^{-8}$ & 6.70 \\
160 & 80 & 32 & $2.08\times 10^{-9}$ & $1.35\times 10^{-8}$ & 6.51 \\
240 & 120 & 48 & $4.21\times 10^{-10}$ & $2.66\times 10^{-9}$ & 6.32 \\
\bottomrule
\end{tabular}
\end{table}

\begin{figure}[htbp]
\centering
\includegraphics[width=0.72\textwidth]{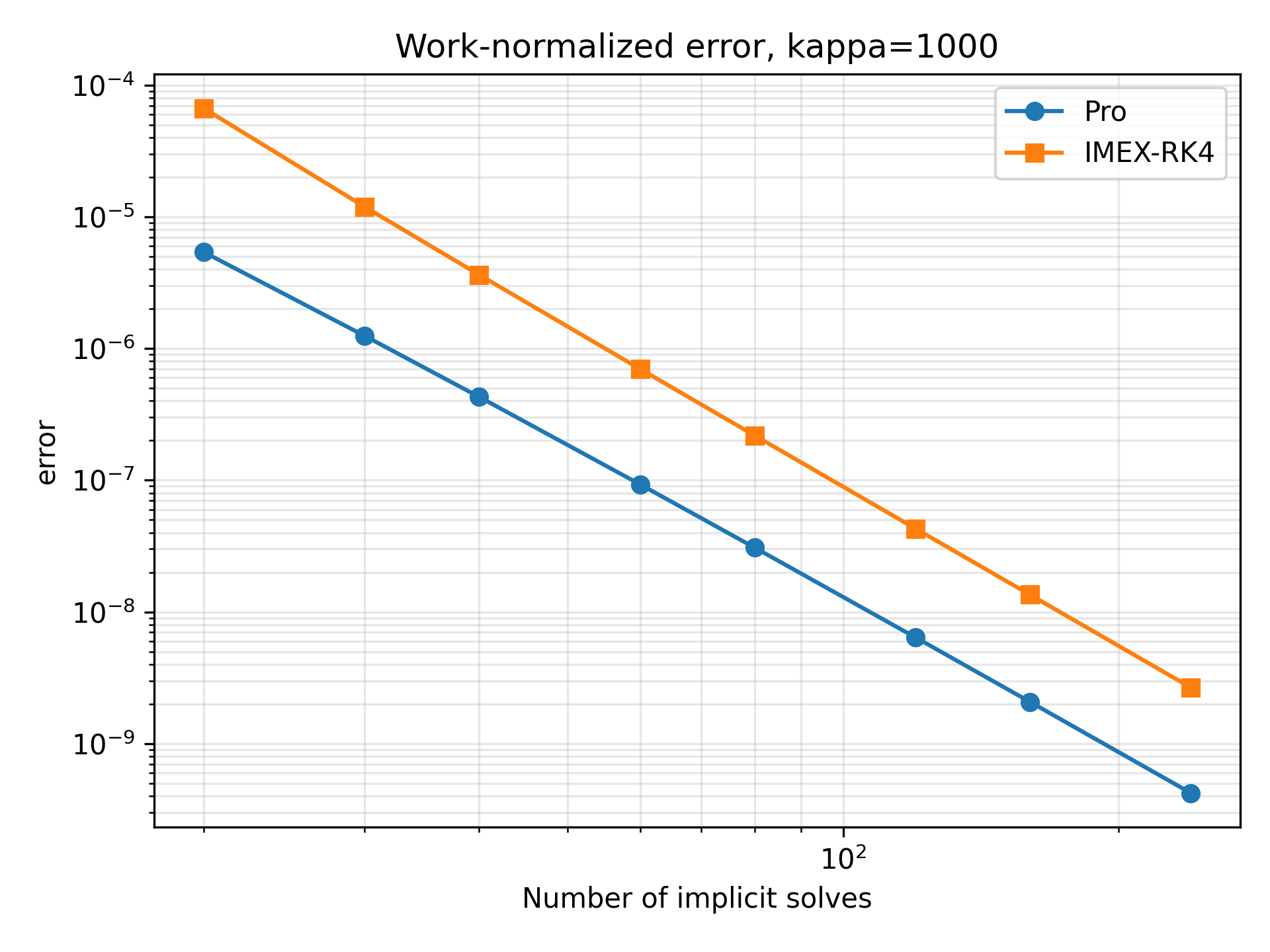}
\caption{Work-normalized error comparison for $\kappa=1000$. The horizontal axis is the total number of implicit solves.}
\label{fig:work_normalized_errors}
\end{figure}

\begin{figure}[htbp]
\centering
\includegraphics[width=0.72\textwidth]{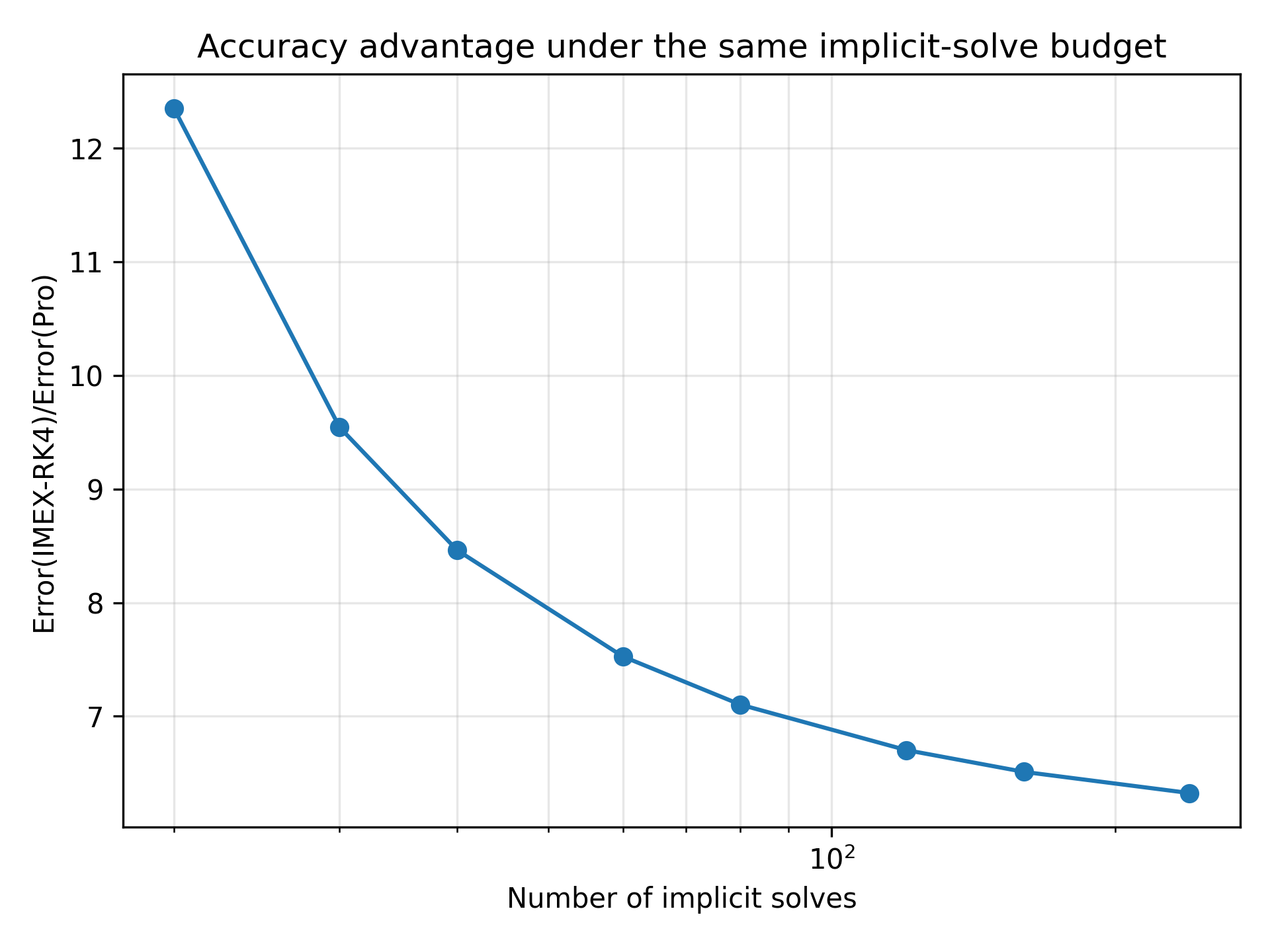}
\caption{Accuracy advantage under the same implicit-solve budget for $\kappa=1000$.}
\label{fig:work_normalized_ratio}
\end{figure}

The proposed method consistently gives smaller errors under the same implicit-solve budget. This supports the claim that the compact two-stage structure improves accuracy per implicit solve for stiff split systems.

\subsection{A practical stiff PDE example: increasing high-frequency diffusion}
\label{subsec:advdiff_increasingK}

As a more practical stiff example, we consider the periodic advection--diffusion equation
\begin{equation}
    u_t+a u_x=\nu u_{xx},
    \qquad x\in[0,2\pi].
\end{equation}
The advection part is treated explicitly and the diffusion part implicitly. For a Fourier mode $\exp(iKx)$, the split eigenvalues are
\begin{equation}
    \lambda_E=-iaK,
    \qquad
    \lambda_I=-\nu K^2.
\end{equation}
Thus, increasing the wave number $K$ strengthens the implicit diffusive stiffness quadratically. We take $a=0.2$, $\nu=1$, $T=0.02$, and a high-frequency amplitude $10^{-2}$. The comparison is performed under a fixed budget of 40 implicit solves.

\begin{table}[htbp]
\centering
\caption{One-dimensional advection--diffusion high-mode test with increasing wave number $K$ under the same budget of 40 implicit solves.}
\label{tab:advdiff_1d_increasingK}
\begin{tabular}{ccccc}
\toprule
$K$ & $|\lambda_I|\Delta t_{\rm KC}$ & Error Pro & Error KC--ARK4 & Ratio \\
\midrule
8 & 0.16 & $3.13\times 10^{-10}$ & $2.00\times 10^{-9}$ & 6.37 \\
12 & 0.36 & $3.40\times 10^{-9}$ & $2.34\times 10^{-8}$ & 6.89 \\
16 & 0.64 & $5.86\times 10^{-9}$ & $4.51\times 10^{-8}$ & 7.70 \\
20 & 1.00 & $2.74\times 10^{-9}$ & $2.45\times 10^{-8}$ & 8.91 \\
24 & 1.44 & $4.44\times 10^{-10}$ & $4.83\times 10^{-9}$ & 10.87 \\
28 & 1.96 & $2.83\times 10^{-11}$ & $4.19\times 10^{-10}$ & 14.82 \\
32 & 2.56 & $7.61\times 10^{-13}$ & $2.15\times 10^{-11}$ & 28.23 \\
\bottomrule
\end{tabular}
\end{table}

\begin{figure}[htbp]
\centering
\includegraphics[width=0.72\textwidth]{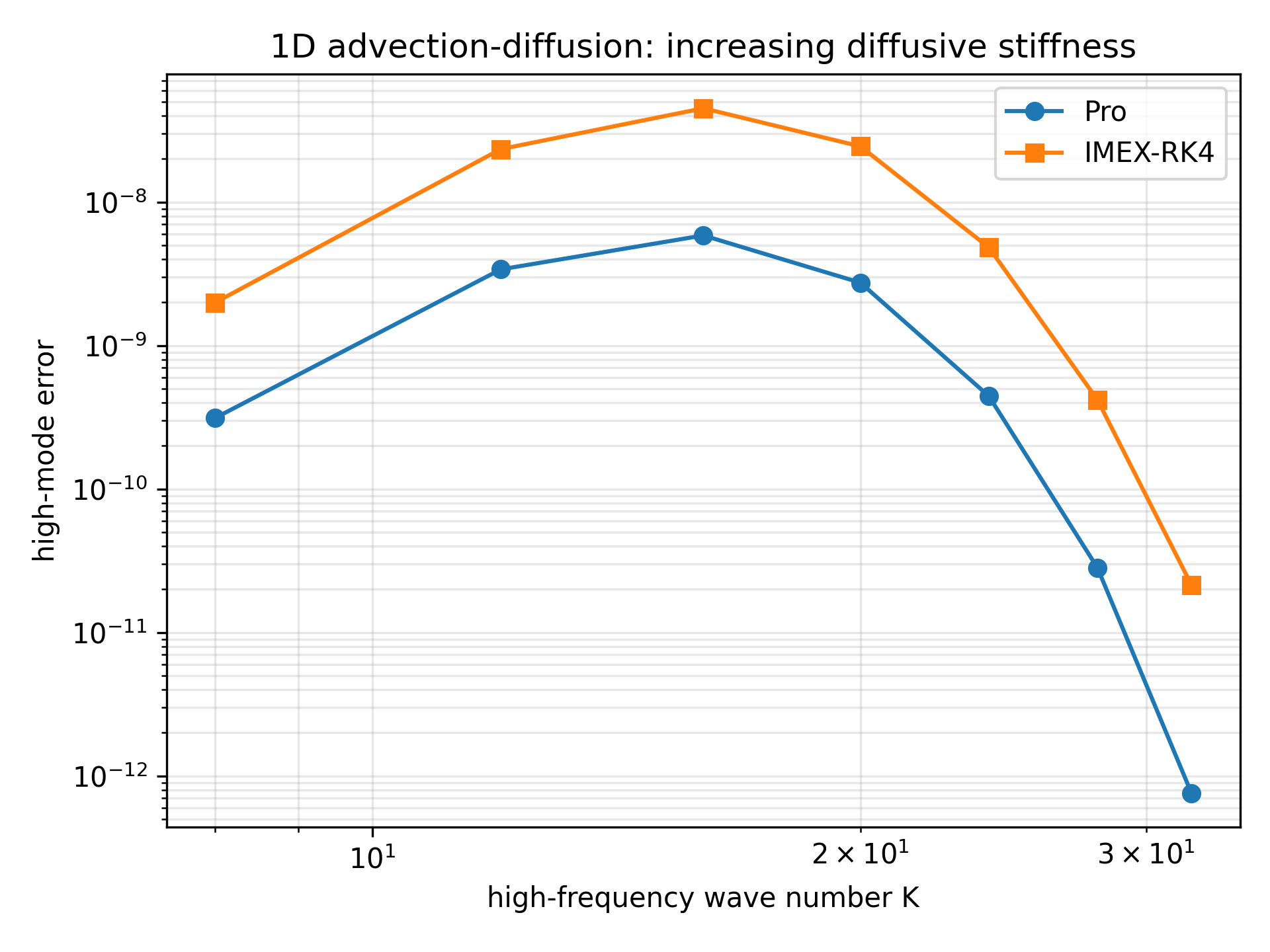}
\caption{High-mode error for the one-dimensional advection--diffusion equation as the wave number $K$ increases.}
\label{fig:advdiff_1d_errors}
\end{figure}

\begin{figure}[htbp]
\centering
\includegraphics[width=0.72\textwidth]{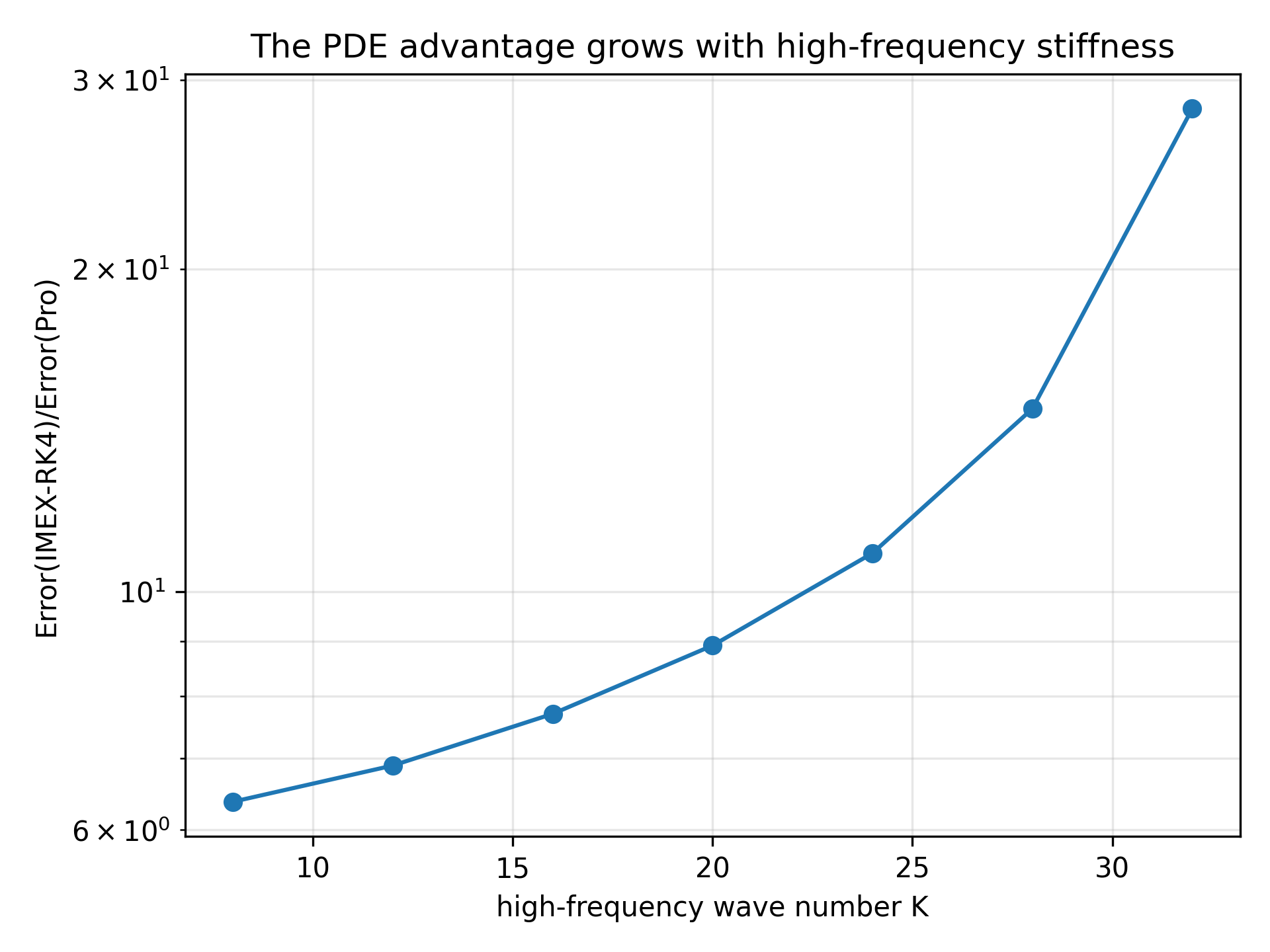}
\caption{The PDE error ratio increases with the high-frequency stiffness. This confirms that the stiff-decay advantage is also visible in a classical advection--diffusion problem.}
\label{fig:advdiff_1d_ratio}
\end{figure}

Table~\ref{tab:advdiff_1d_increasingK} and Figures~\ref{fig:advdiff_1d_errors}--\ref{fig:advdiff_1d_ratio} demonstrate that the advantage of Pro becomes clearer as the diffusive high-frequency mode becomes stiffer. This experiment is important because it shows that the stronger damping mechanism is not limited to scalar model equations or small matrix systems, but also appears in a standard split PDE setting.

\subsection{Two-dimensional advection--diffusion high-mode test}
\label{subsec:advdiff_2d}

Finally, we consider the two-dimensional periodic advection--diffusion equation
\begin{equation}
    u_t+a u_x+b u_y=\nu(u_{xx}+u_{yy}),
    \qquad (x,y)\in[0,2\pi]^2.
\end{equation}
For the Fourier mode $\exp(i(k_xx+k_yy))$, the split eigenvalues are
\begin{equation}
    \lambda_E=-i(ak_x+bk_y),
    \qquad
    \lambda_I=-\nu(k_x^2+k_y^2).
\end{equation}
We take $a=b=0.2$, $\nu=0.5$, $(k_x,k_y)=(32,24)$, and $T=0.02$. The comparison is again normalized by the number of implicit solves.

\begin{table}[htbp]
\centering
\caption{Two-dimensional advection--diffusion high-mode test under the same implicit-solve budgets.}
\label{tab:advdiff_2d_budget}
\begin{tabular}{ccccccc}
\toprule
Budget & $N_{\rm Pro}$ & $N_{\rm KC}$ & Error Pro & Error KC--ARK4 & Ratio & $|\lambda_I|\Delta t_{\rm KC}$ \\
\midrule
20 & 10 & 4 & $2.27\times 10^{-10}$ & $7.31\times 10^{-8}$ & 322.15 & 4.00 \\
30 & 15 & 6 & $6.07\times 10^{-11}$ & $1.70\times 10^{-9}$ & 27.97 & 2.67 \\
40 & 20 & 8 & $2.25\times 10^{-11}$ & $3.44\times 10^{-10}$ & 15.30 & 2.00 \\
60 & 30 & 12 & $5.29\times 10^{-12}$ & $5.52\times 10^{-11}$ & 10.43 & 1.33 \\
80 & 40 & 16 & $1.84\times 10^{-12}$ & $1.65\times 10^{-11}$ & 8.96 & 1.00 \\
120 & 60 & 24 & $4.02\times 10^{-13}$ & $3.14\times 10^{-12}$ & 7.79 & 0.67 \\
160 & 80 & 32 & $1.34\times 10^{-13}$ & $9.79\times 10^{-13}$ & 7.28 & 0.50 \\
240 & 120 & 48 & $2.81\times 10^{-14}$ & $1.91\times 10^{-13}$ & 6.81 & 0.33 \\
\bottomrule
\end{tabular}
\end{table}

\begin{figure}[htbp]
\centering
\includegraphics[width=0.72\textwidth]{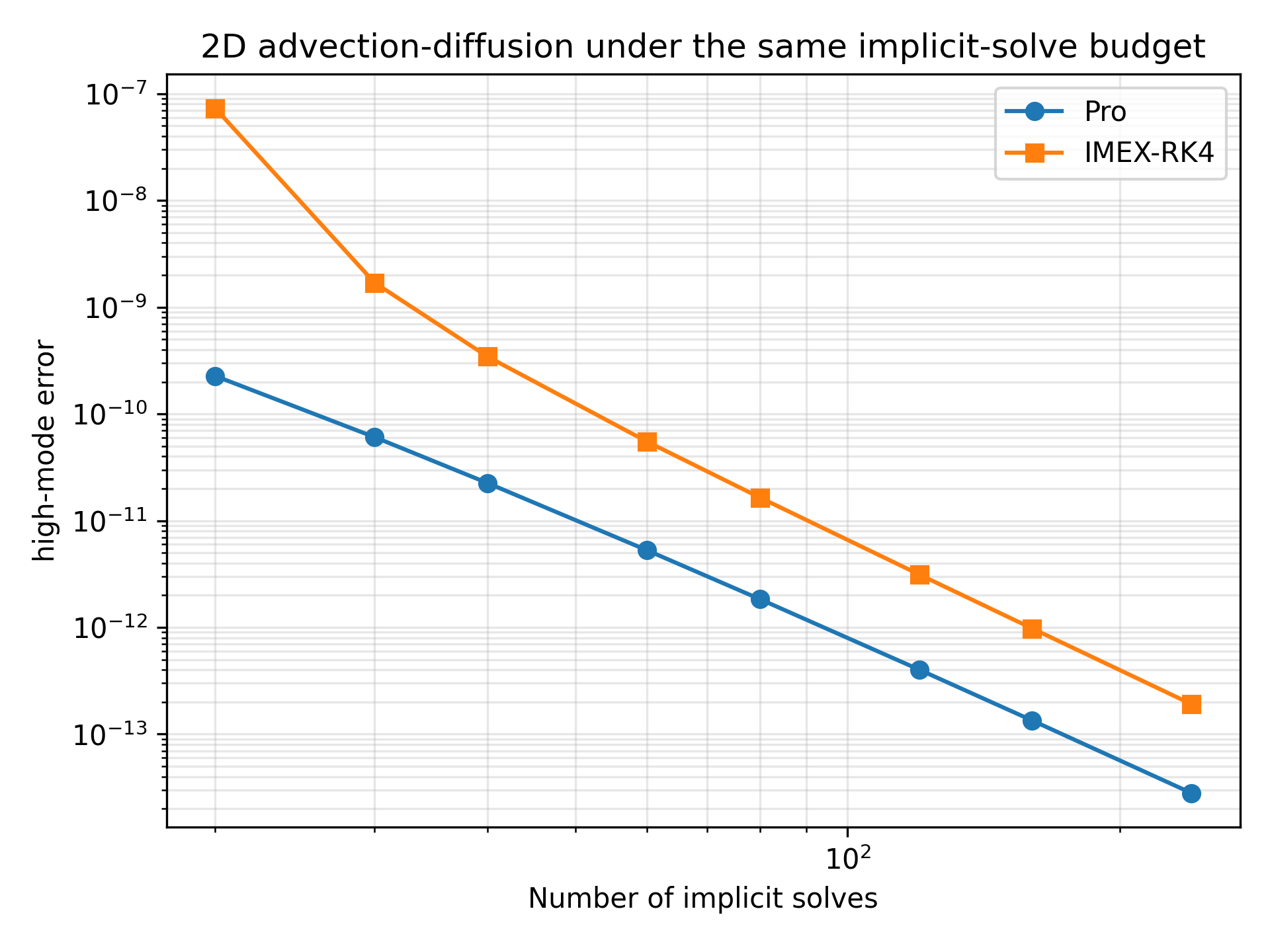}
\caption{Two-dimensional advection--diffusion high-mode error under the same implicit-solve budgets.}
\label{fig:advdiff_2d_budget}
\end{figure}

The two-dimensional test further confirms that the proposed method can produce smaller high-mode errors under the same implicit-solve budget in stiff split PDE problems. The advantage is strongest for coarse budgets, where the effective implicit stiffness number $|\lambda_I|\Delta t_{\rm KC}$ is larger.

\subsection{Summary of the numerical evidence}
\label{subsec:numerical_summary}

The numerical results support the conclusions of Section~\ref{sec:discussion_comparison}. First, the mixed-consistency test shows that the full-vector-field derivative is essential: once the mixed derivative terms are omitted, high-order accuracy is lost. Second, the scalar stiff-mode test verifies the stronger $O(z^{-2})$ stiff decay of the proposed method, in contrast to the $O(z^{-1})$ decay of the KC--ARK4 reference method. Third, the increasing-stiffness matrix test and the one- and two-dimensional high-frequency advection--diffusion tests show that the advantage of the proposed method becomes more pronounced as the effective stiff modes become stronger. Finally, the work-normalized comparisons demonstrate that the compact two-stage structure leads to smaller errors under the same implicit-solve budget. Therefore, the main advantage of the proposed method lies in stiff multiscale regimes where mixed consistency, stiff-mode damping, and the cost of implicit solves are all important.


\section{Conclusion}
\label{sec:conclusion}

This paper constructs and analyzes a fourth-order IMEX-like two-derivative time discretization method for additively split evolution systems. Unlike classical multi-stage fourth-order implicit--explicit Runge--Kutta methods, which usually rely on several stages and complicated explicit--implicit coupling order conditions to achieve high-order accuracy, the proposed method follows a different construction principle. It employs not only the split physical quantities, but also their temporal derivatives evaluated along the full vector field, thereby naturally incorporating explicit terms, implicit terms, and mixed coupling information within the scheme. Based on this construction, the proposed method achieves fourth-order temporal accuracy through a compact Hermite-type integral formulation.

Compared with classical multi-stage fourth-order IMEX--RK schemes, the proposed scheme has several advantages. First, it possesses inherent mixed consistency: the coupling between the explicit and implicit components is not enforced through additional complicated order conditions, but is naturally reflected through the temporal derivatives evaluated along the complete vector field. Second, the method has a compact structure. By using physical quantities together with their temporal derivatives, it attains fourth-order accuracy without relying on numerous stages, coefficients, or intricate coupling constraints, thus reducing the implementation burden associated with classical multi-stage fourth-order IMEX--RK schemes and improving the simplicity of both formulation and computation. Third, for stiff problems, the implicit two-derivative structure enhances the damping of stiff modes in the purely implicit limit and in the strong implicit-stiffness limit with a fixed explicit component. In addition, the method admits a clear Hermite-type integral interpretation, remains consistent with TSFO-type methods in the non-stiff limit, and exhibits favorable stiff decay behavior in the purely implicit scalar limit.

Nevertheless, the proposed method also has some limitations. Since temporal derivative terms are involved in the scheme, these derivatives must be either analytically evaluated or numerically approximated in practical computations. Moreover, the method does not belong to the strict framework of classical IMEX--RK methods. A complete characterization of its two-dimensional stability region, especially the stability structure under the combined effects of the explicit and implicit stability variables, still requires further investigation.

Numerical experiments verify the main theoretical conclusions of this paper. In the non-commuting linear split system, the proposed full-vector-field formulation retains fourth-order convergence, whereas the incomplete derivative variant loses high-order accuracy, confirming the necessity of the mixed derivative terms. The purely implicit scalar test verifies the predicted $O(z^{-2})$ stiff decay and the stronger damping of stiff modes compared with the KC--ARK4 reference method. The increasing-stiffness matrix tests and the one- and two-dimensional high-frequency advection--diffusion tests further show that the proposed method produces smaller errors under the same implicit-solve budgets, and that this advantage becomes more pronounced when the effective stiff modes become stronger. Overall, the proposed IMEX-like two-derivative scheme provides a compact, physically transparent high-order time discretization approach with strong stiff-mode damping in the above limiting regimes, particularly in situations where temporal derivative information is available or can be efficiently approximated.

\section*{Acknowledgments}

This work was supported by the Key Program of Henan Higher Education Institutions (Grant No. 26A110007), the Young Talents Fund of Henan Province (Grant No. 252300423500), the Double First-Class Project of the School of Geomatics of Henan Polytechnic University (Grant No. BSJJ202306), and the Doctoral Startup Foundation of Henan Polytechnic University (Grant No. B2024-60).


\end{document}